\documentclass[a4paper,english 12pt]{amsart}

\usepackage[utf8]{inputenc}
\usepackage[T1]{fontenc}
\usepackage[english]{babel}
\usepackage{multicol}
\usepackage{enumitem}
\usepackage{pifont}
 \usepackage{pgf,tikz,pgfplots}
\pgfplotsset{compat=1.15}
\usepackage{mathrsfs}
\usetikzlibrary{arrows}
\usepackage{mathtools}
\usepackage{float}
\usepackage{appendix}
\usepackage{tikz}
\usepackage{xcolor}

\usepackage{hyperref}

\usepackage{amsmath,amssymb}
\usepackage{mathrsfs}
\usepackage{amsthm}
\usepackage{yhmath}
\usepackage{dsfont}
\numberwithin{equation}{section}

\usepackage{graphicx}
\usepackage{subcaption}

\usepackage[left=3cm,right=3cm,top=2.5cm,bottom=3cm]{geometry}

\DeclareMathOperator{\im}{Im}
\DeclareMathOperator{\re}{Re}

\DeclareMathOperator{\Id}{Id}

\DeclareMathOperator{\I}{I}

\newcommand{\R}{\mathbb{R}}

\newcommand{\N}{\mathbb{N}}

\newcommand{\C}{\mathbb{C}}

\DeclareMathOperator{\cinf}{\emph{C}^\infty}
\DeclareMathOperator{\cinfc}{\emph{C}_c^\infty}
\DeclareMathOperator{\supp}{supp}

\DeclareMathOperator{\op}{Op_{\textit{h}}}
\DeclareMathOperator{\Gr}{Gr}
\DeclareMathOperator{\WF}{WF_{\textit{h}}}
\newcommand{\hinf}{O(h^\infty)}

\newtheorem{thm}{Theorem}
\newtheorem*{thm*}{Theorem}

\theoremstyle{definition}

\newtheorem{defi}{Definition}[section]
\newtheorem{prop}{Proposition}[section]

\newtheorem{lem}{Lemma}[section]
\newtheorem{cor}{Corollary}[section]
\newtheorem*{rem}{Remark}

\makeatletter
\def\paragraph{\vspace{0.4cm} \@startsection{paragraph}{4}%
  \z@\z@{-\fontdimen2\font}%
  {\normalfont\bfseries}}
\makeatother

\makeatletter
\def\subparagraph{\vspace{0.3cm} \@startsection{subparagraph}{4}%
  \z@\z@{-\fontdimen2\font}%
  {\normalfont\bfseries}}
\makeatother

\title[Resolvent estimates in strips for obstacle scattering in 2D]{Resolvent estimates in strips for obstacle scattering in 2D and local energy decay for the wave equation}
\author{Lucas Vacossin}
\address{Universit\'e Paris-Saclay, Laboratoire de mathématiques d'Orsay, UMR 8628 du CNRS, B\^atiment 307, 91405 Orsay Cedex,}
\email{lucas.vacossin@universite-paris-saclay.fr}

\makeindex
\begin{document}
\maketitle

\begin{abstract}
In this note, we are interested in the problem of scattering by $J$ strictly convex obstacles satisfying a no-eclipse condition in dimension 2. We use the result of \cite{Vacossin} to obtain polynomial resolvent estimates in strips below the real axis. We deduce estimates in $O(|\lambda| \log |\lambda|)$ for the truncated resolvent on the real line and give an application to the decay of the local energy for the wave equation. 
\end{abstract}

\section{Introduction}
\subsection{Spectral gap and resolvent estimates.}

Let $(\mathcal{O}_j)_{1 \leq j \leq J}$ be open, strictly convex obstacles in $\R^2$ having smooth boundary and satisfying the \emph{Ikawa condition} of no-eclipse: for $i \neq j \neq k$, $\overline{\mathcal{O}_i}$ does not intersect the convex hull of $ \overline{ \mathcal{O}_j }\cup \overline{\mathcal{O}_k}$. Let  $$\mathcal{O} = \bigcup_{j=1}^J \mathcal{O}_j \; ; \; 
\Omega = \R^2  \setminus \overline{\mathcal{O}}. $$
\begin{figure}[h]
\begin{center}
\includegraphics[width=12cm]{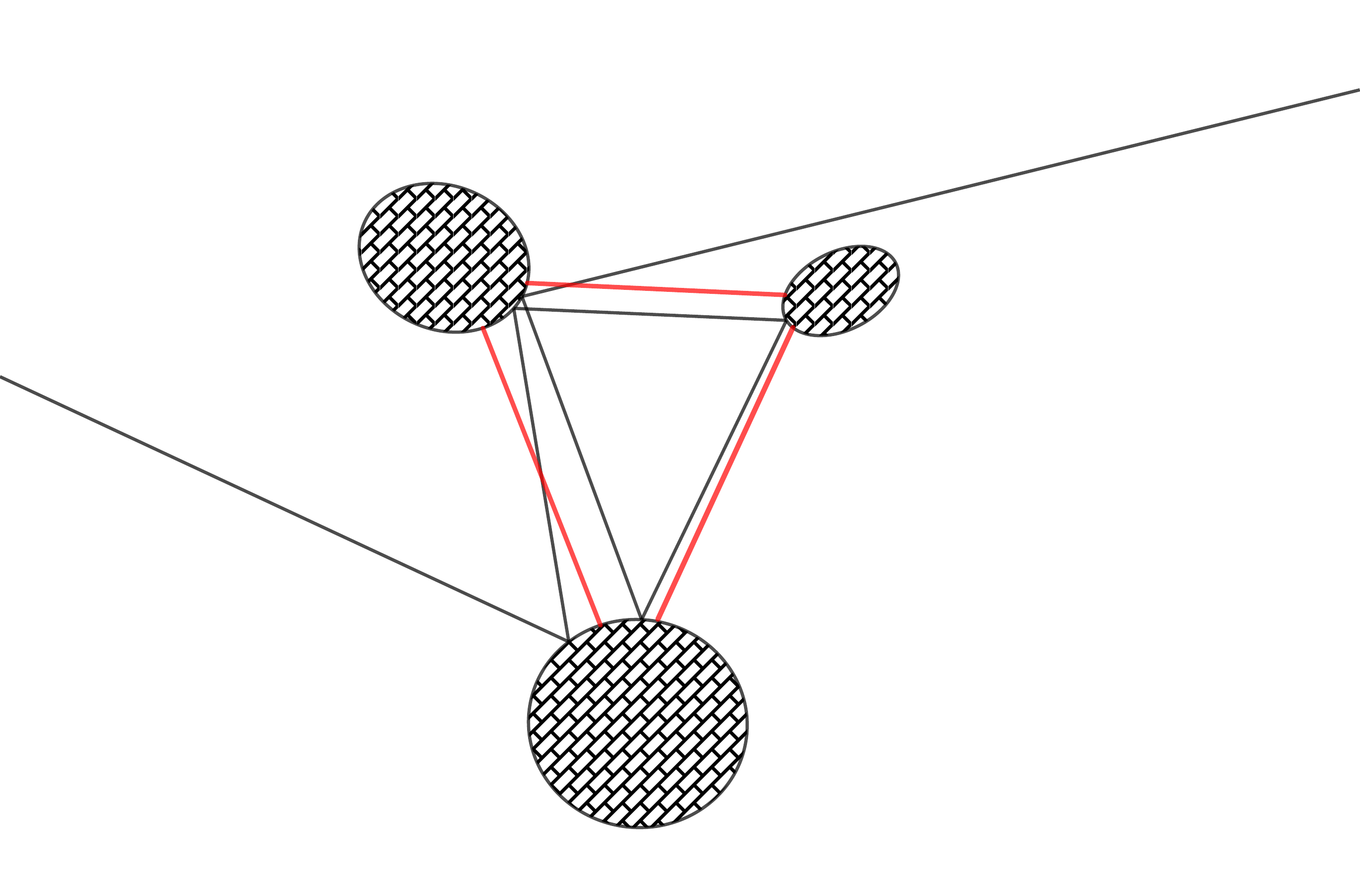}
\caption{Scattering by three obstacles in the plane}
\end{center}
\end{figure}

It is known that the resolvent $R(\lambda)$ of the Dirichlet Laplacian $-\Delta_\Omega$ in $\Omega$ continues meromorphically to the logarithmic cover of $\C$ (see for instance \cite{DyZw}, Chapter 4, Theorem 4.4).  More precisely, if $\chi \in \cinfc(\R^2)$ is equal to one in a neighborhood of $\overline{\mathcal{O}}$, 
\begin{equation}\label{resolvant}
\chi R(\lambda)\chi=\chi (-\Delta_\Omega - \lambda^2 )^{-1} \chi  : L^2 (\Omega) \to L^2(\Omega)
\end{equation}
is holomorphic in the region $\{ \im \lambda >0 \}$ and it continues meromorphically to the logarithmic cover $\Lambda$ of $\C$. Its poles are the \emph{scattering resonances} and do not depend on $\chi$. In \cite{Vacossin}, the following result has been proved :  

\begin{thm}\label{Thm1}
There exist $\gamma >0$ and $\lambda_0 >0$ such that there is no resonance in the region 
$$ [ \lambda_0, + \infty [ +i[- \gamma , 0 ]$$
seen as a region in the first sheet of $\Lambda$. 
\end{thm}

In this note, we reuse the arguments of \cite{NSZ14} and the main estimate in \cite{Vacossin} (Proposition 4.1) to obtain estimates for the cut-off resolvent (\ref{resolvant}) in this region. We will rather state these resolvent estimates in a semiclassical form, so that it can also be applied to more general semiclassical problems such as the scattering by a smooth compactly potential (see \cite{NSZ11}, Section 2, for precise assumptions and \cite{Vacossin}, Section 2.2 for applications of Theorem \ref{Thm1} in scattering by a potential under these assumptions). 
In obstacle scattering, the semiclassical problem is simply a rescaling : we are interested in the semiclassical operator 
$$ P(h)  = -h^2 \Delta_\Omega - 1 \quad , \quad 0< h \leq h_0$$ and spectral parameter $z \in [-\delta,\delta] + i [-Kh,Kh]$ for some fixed $K >0$ and some $\delta >0$.  We note 
\begin{equation}\label{semiclassical_resolvent}
R_h(z) = (P(h) -z)^{-1}
\end{equation} continued meromorphically from $\im z >0$ to $z \in  [-\delta,\delta] + i [-Kh,Kh]$. We prove :

\begin{thm}\label{Thm2}
Suppose that $P(h) = -h^2 \Delta_\Omega -1$  where $-\Delta_\Omega $ is the Dirichlet Laplacian in $\Omega$, or $P(h) = -h^2 \Delta +V - E_0$ where $V \in \cinfc(\R^2)$ and $E_0 \in \R^*_+$ satisfying the assumptions of \cite{NSZ11}, recalled in \ref{Section_potential_scattering}. Let $\chi \in \cinfc(\R^2)$ be equal to one in a neighborhood of $\overline{\mathcal{O}}$ (in the case of obstacle scattering) or $\supp V$ (in the case of scattering by a potential). Fix $K >0$. There exists $\delta_0 >0$, $\gamma >0$, $C>0$, $h_0 >0$ and $\beta \geq 0$ such that for all $0 < h \leq h_0$, $P(h)$ has no resonance in 
\begin{equation} 
\mathcal{D}_h \coloneqq  \{ z \in \C, \re z \in [-\delta_0, \delta_0], -\gamma h \leq \im z \leq K h  \}  
\end{equation} and for all $z \in \mathcal{D}_h$, 
\begin{equation}\label{eq_thm_2}
|| \chi R_h(z) \chi ||_{L^2 \to L^2} \leq C h^{-\beta} 
\end{equation}
\end{thm}

\begin{rem}
In the case of the obstacles, with these notations, for $\delta$ small enough and $h$ small enough, $z$ is related to the spectral parameter $\lambda_h(z)$ by the relation $\lambda_h(z)^2 = h^{-2}(1+z)$. As a consequence, $\lambda_h(z)$ lies in a neighborhood of $1/h$ in $\Lambda$. In particular, it lives in the first sheet of $\Lambda$, that is $\arg \lambda_h(z) \in ]-\pi/2, \pi/2[$. \\
\end{rem}

\subsection{Applications }

\paragraph{Decay of the local energy for the wave equation. }
As a first application, we obtain a decay rate $O(t^{-2})$ for the local energy of the wave equation outside the obstacles. 
 The link between resolvent estimates and energy decay is quite standard now (see for instance \cite{ZW}, Chapter 5, \cite{Lebeau96}). In the particular case of obstacle scattering, Ikawa showed exponential decay \emph{in dimension 3}, for the case of two obstacles (\cite{Ik82}) and for more obstacles under a dynamical assumption (\cite{Ik88}) involving the topological pressure $P(s)$ of the billiard flow. This assumption requires the pressure to be strictly negative at $s=1/2$ (see also \cite{NZ09}). In the case of dimension 2 (and more generally, of even dimensions), one cannot expect such an exponential decay, due to the logarithmic singularity at 0 for the free resolvent and the fact that the strong Huygens principle does not hold.  Even for the free case, the bound for the local energy is $O(t^{-2})$. This is the bound we obtain here, assuming that the initial data are sufficiently regular : 
 
 \begin{thm}\label{Thm3}
 There exists $k \in \N$ such that for all $R>0$, there exists $C_R >0$ such that the following holds: let $(u_0,u_1) \in \left( H^{k+1}(\Omega) \cap H_0^1(\Omega) \right) \times H^k(\Omega)$ be initial data supported in $B(0,R) \cap \Omega$ and consider the unique solution of the Cauchy problem 
 $$ \left\{ \begin{array}{l}
  \partial^2_t u(t,x) - \Delta u(t,x) = 0 \text{ in } \Omega \\
  u(t,x) = 0 \text{ on } \partial \Omega \\
  u_{t=0} = u_0 \\
  \partial_t u_{t=0} = u_1 
\end{array}\right.$$
Then, for $t \geq 1$, the local energy in the ball $B(0,R)$, $E_R(t)$, satisfies the bound $$E_R(t) \coloneqq \int_{B(0,R) \cap \Omega} |\nabla u(t)|^2 + |\partial_t u(t)|^2 \leq \frac{C_R}{t^2} \left( ||u_0||^2_{H^{k+1}} + ||u_1||^2_{H^k } \right) $$
 \end{thm}
 
 Theorem \ref{Thm3} is a consequence of Theorem \ref{Thm2}. This fact is proved in Section \ref{Section_decay_energy} and the proof uses the strategy of \cite{Burq99}. 
 
 \paragraph{Resolvent estimates on the real line. }
 Polynomial resolvent bounds in strips are known to imply better bounds on the real line, by using a semiclassical maximum principle (see for instance \cite{Burq04}, Lemma 4.7, or  \cite{Ing17}). As a consequence, we deduce the following estimates on the real line : 

\begin{cor}
Let $P(h)$ be one of the operators described in Theorem \ref{Thm2} and let $\chi \in \cinfc(\R^2)$ as in this Theorem. 
There exits $C_0 >0$, $\delta_0 >0$ and $h_0 >0$ such that for all $0 < h \leq h_0$ and for all $z \in [-\delta_0, \delta_0]$, 
$$|| \chi (P(h) - z )^{-1} \chi ||_{L^2 \to L^2} \leq C_0 \frac{| \log h|}{h}$$
\end{cor}

\begin{rem}
As a direct corollary of the proof of Lemma 4.7 in \cite{Burq04}, we can obtain a more general bound : for $h>0$ small enough, 
\begin{equation}\label{interpolation_estimate}
|| \chi (P(h) - z )^{-1} \chi ||_{L^2 \to L^2} \leq C_0| \log h| h^{-1 + \sigma|\im z|/h}  \, , \,z \in [-\delta_0, \delta_0] +i[-\gamma h, 0]
\end{equation}
where $\sigma >0$. With this method, based on the maximum principle for analytic functions, the value of $\sigma$ in not explicit. In fact, our proof gives a bound of the form 
$$||\chi (P(h) - z )^{-1} \chi ||_{L^2 \to L^2} \leq C| \log h| h^{-1 - M_1 - M_2 |\im z|/h}  \, , \,z \in [-\delta_0, \delta_0] +i[-\gamma h, 0] $$
where $M_2$ only depends on constants related to the billiard map (see (\ref{last_equation})). The extra $M_1$ is a consequence of the method we use, based on the use of an escape function. It is possible that a more careful analysis could allow to get rid of this extra $M_1$ and we could straighlty obtain a bound of the form (\ref{interpolation_estimate}). 
\end{rem}

This kind of estimates is known to be useful to prove smoothing effects for the Schrödinger equation and to obtain Strichartz estimates, which turns out to be crucial for the local-well posedness of the non-linear Schrödinger equation (see for instance \cite{Burq04}, \cite{BuGeTz}). Let's for instance mention the following smoothing estimates (see the references above for the proof and for pointers to the literature concerning these estimates) : 

\begin{cor}
Let $\Omega$ be as in Theorem \ref{Thm2} and let $e^{-i t \Delta_\Omega}$ be the Schrödinger propagator of the Dirichlet Laplacian $-\Delta_\Omega$ in $\Omega$. Then, for any $\varepsilon>0$ and for any $\chi \in \cinfc(\R^2)$ equal to one in a neighborhood of $\overline{\mathcal{O}}$, there exists $C>0$ such that for any $u_0 \in L^2(\Omega)$, 
$$ || \chi e^{-it \Delta_\Omega} \chi u_0 ||_{L^2 (\R_t, H^{1/2-\varepsilon}(\Omega))} \leq C ||\chi u_0||_{L^2}$$ 
\end{cor}

\paragraph{Organization of the paper. } In Section \ref{Section_2}, we prove Theorem \ref{Thm2} using the crucial estimate proved in \cite{Vacossin} and recalling the reduction to open quantum maps performed in \cite{NSZ11} and \cite{NSZ14}. The main semiclassical ingredients of the above paper are recorded in Appendix \ref{Appendix}. Section \ref{Section_decay_energy} is devoted to the proof of the local energy decay for the wave equation. 

\subsubsection*{Acknowledgment}
The author would like to thank Stéphane Nonnenmacher for his careful reading of a preliminary version of this work and Maxime Ingremeau for suggesting to write this note.

\section{Proof of Theorem \ref{Thm2}}\label{Section_2}

In this section, we prove the main resolvent estimate of this note. The central point, concerning a resolvent bound for open hyperbolic quantum maps, is common to the case of obstacle scattering and scattering by a potential. However, the reduction to open quantum maps differs in the two above situations, this is why we distinguish the two cases. We begin by recalling the definitions of open quantum maps from \cite{NSZ11} and \cite{NSZ14} and the result of \cite{Vacossin} leading to a crucial resolvent bound. 

\subsection{Resolvent bound for open quantum maps. }

\subsubsection{Definitions. }
The following long definition is based on the definitions in the works of Nonnenmacher, Sjöstrand and Zworski in \cite{NSZ11} and \cite{NSZ14} specialized to the 2-dimensional phase space. 
Consider open intervals $Y_1, \dots, Y_J$ of $J$ copies of $\R$ and set  : 
$$Y = \bigsqcup_{j=1}^J Y_j \subset \bigsqcup_{j=1}^J \R$$
and consider 
$$ U = \bigsqcup_{j=1}^J U_j \subset \bigsqcup_{j=1}^J T^*\R^d \quad ; \quad \text{where } U_j \Subset T^*Y_j \text{ are open sets }$$
The Hilbert space $L^2(Y)$ is the orthogonal sum $\bigoplus_{i=1}^J L^2(Y_i)$. 

For $j=1, \dots, J$, consider open disjoint subsets $\widetilde{D}_{i j } \Subset U_j$, $1 \leq i \leq J$, \emph{the departure sets}, and similarly, for $i= 1, \dots, J$ consider open disjoint subsets $\widetilde{A}_{i j } \Subset U_i$, $1 \leq j \leq J$, \emph{the arrival sets} (see Figure \ref{figure_example_1}). We assume that there exist smooth symplectomorphisms 

\begin{figure}
\centering
\includegraphics[scale=0.5]{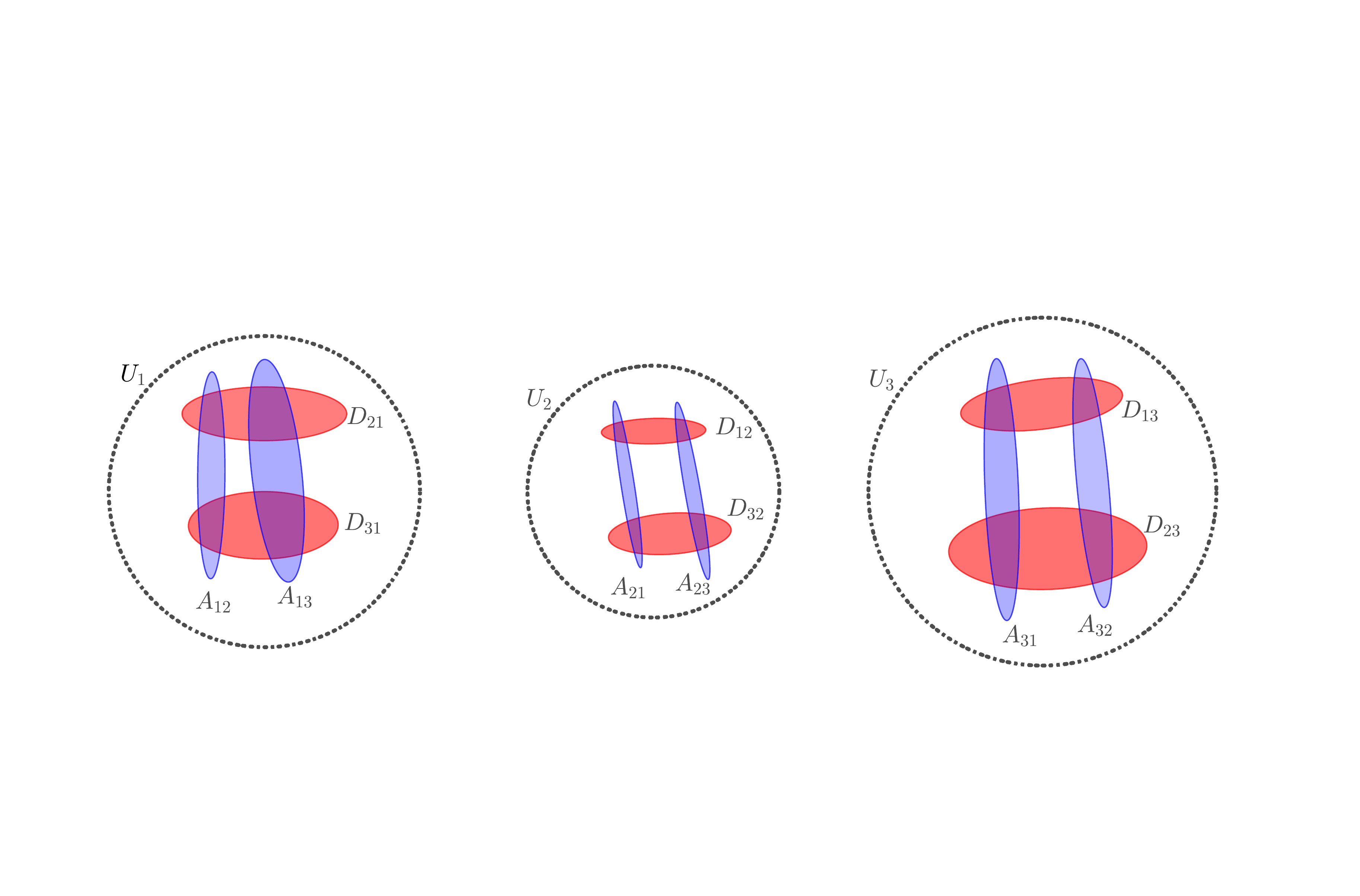}
\caption{A schematic example with $J=3$  in a case where $D_{ii} = \emptyset$ for $i=1,2,3$.}
\label{figure_example_1}
\end{figure}

\begin{equation}
F_{i j } : \widetilde{D}_{i j } \to F_{ij } \left(\widetilde{D}_{i j } \right) = \widetilde{A}_{i j }
\end{equation}
We note $F$ for the global smooth map $F : \widetilde{D} \to \widetilde{A}$ where $\widetilde{A}$ and $\widetilde{D}$ are the full arrival and departure sets, defined as 
$$\widetilde{A} = \bigsqcup_{i=1}^J \bigsqcup_{j=1}^J \widetilde{A}_{i j } \subset \bigsqcup_{i=1}^J U_i $$ 
$$  \widetilde{D} = \bigsqcup_{j=1}^J \bigsqcup_{i=1}^J \widetilde{D}_{i j } \subset \bigsqcup_{j=1}^J U_j $$ 
We define the outgoing (resp. incoming) tail by $\mathcal{T}_+ \coloneqq \{ \rho \in U ; F^{-n}(\rho) \in U , \forall n \in \N \} $ (resp. $\mathcal{T}_- \coloneqq \{ \rho \in U ; F^{n}(\rho) \in U , \forall n \in \N \} $). We assume that they are closed subsets of $U$ and that the \emph{trapped set} 
\begin{equation}\label{trapped_set} 
\mathcal{T} = \mathcal{T}_+ \cap \mathcal{T}_-
\end{equation}
is compact. We also assume that 
\begin{center}
$\mathcal{T}$ is \emph{totally disconnected.}   
\end{center}

For $i , j \in \{1, \dots , J \}$, we note $\mathcal{T}_i = \mathcal{T} \cap U_i$, $$D_{i j } = \{ \rho \in \mathcal{T}_j ; F(\rho) \in \mathcal{T}_i \} \subset  \widetilde{D}_{i j }  $$ and $$
A_{i j } = \{ \rho \in \mathcal{T}_i ; F^{-1}(\rho) \in \mathcal{T}_j\} \subset \widetilde{A}_{ij}$$ 

\begin{rem}
It is possible that for some values of $i$ and $j$, $\widetilde{D}_{ij} = \emptyset$. For instance, when dealing with the billiard map (see (\ref{billiard_map_definition})), the sets $\widetilde{D}_{ii}$ are all empty. 
\end{rem}

We then make the following hyperbolic assumption. 

\begin{equation}\label{Hyperbolicity_assumption}
\mathcal{T} \text{ is a hyperbolic set for }F  
\end{equation}
Namely, for every $\rho \in \mathcal{T}$, we assume that there exist stable/unstable tangent spaces $E^{s}(\rho)$ and $E^{u}(\rho)$ such that : 
\begin{itemize}
\item $\dim E^{s}(\rho) = \dim E^{u}(\rho) = 1$
\item $T_\rho U = E^{s}(\rho) \oplus E^{u}(\rho)$
\item there exists $\lambda >0$, $C >0$  such that for every $v \in T_\rho U$ and any $n \in \N$, 
\begin{align}
v \in E^{s}(\rho)  \implies ||d_\rho F^n (v) || \leq C e^{-n \lambda} || v||  \label{hyp1}\\
v \in E^{u}(\rho) \implies ||d_\rho F^{- n} (v) || \leq C e^{-n \lambda} \label{hyp2} || v||
\end{align}where  $|| \cdot ||$ is a fixed Riemannian metric on $U$. 
\end{itemize}
The decomposition of $T_\rho U$ into stable and unstable spaces is assumed to be continuous. \\

\begin{figure}
\begin{subfigure}{0.45 \textwidth}
\centering
\includegraphics[scale=0.2]{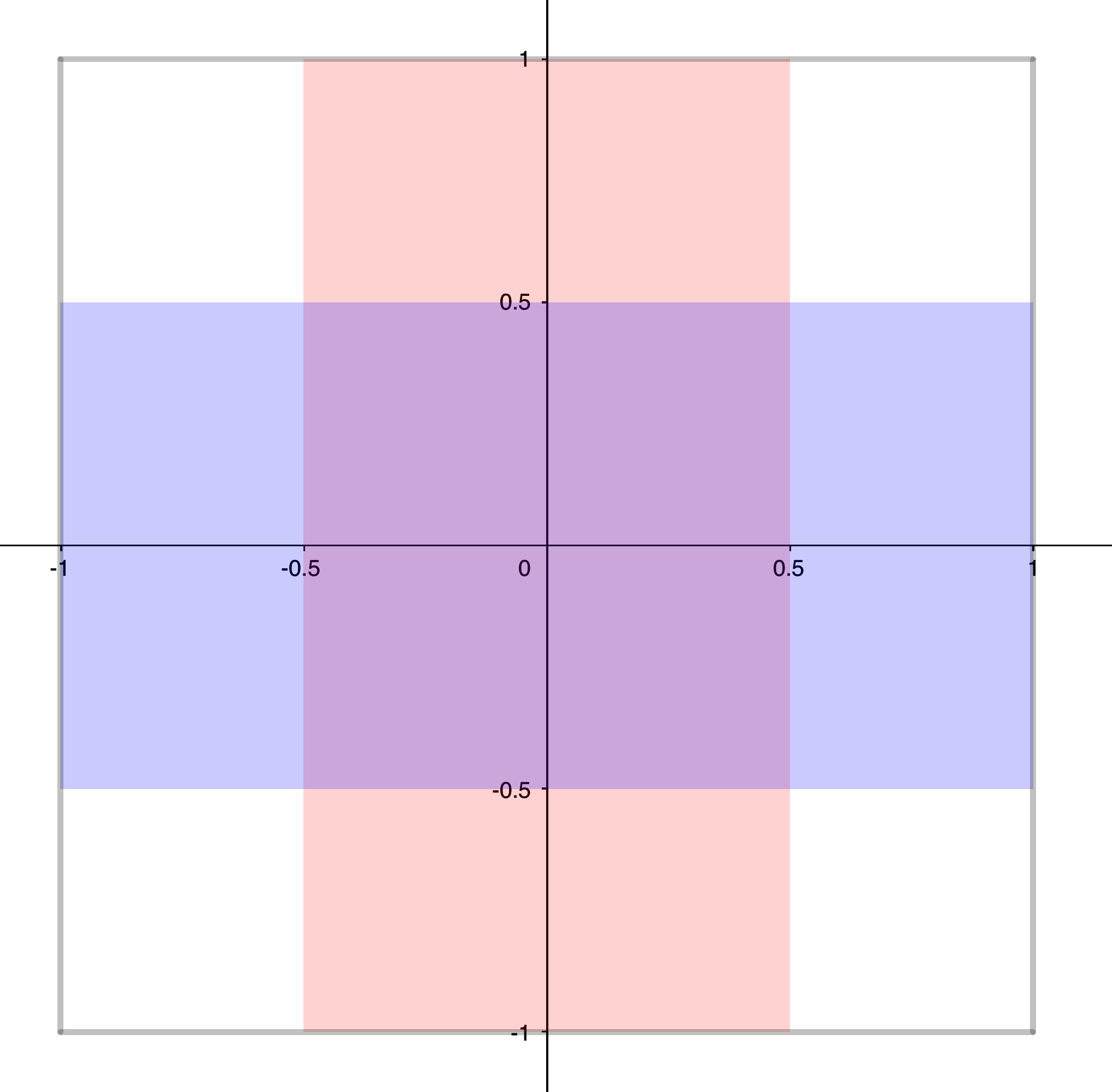}
\caption{Single hyperbolic fixed point.}
\end{subfigure}
\begin{subfigure}{0.45 \textwidth}
\centering
\includegraphics[scale=0.25]{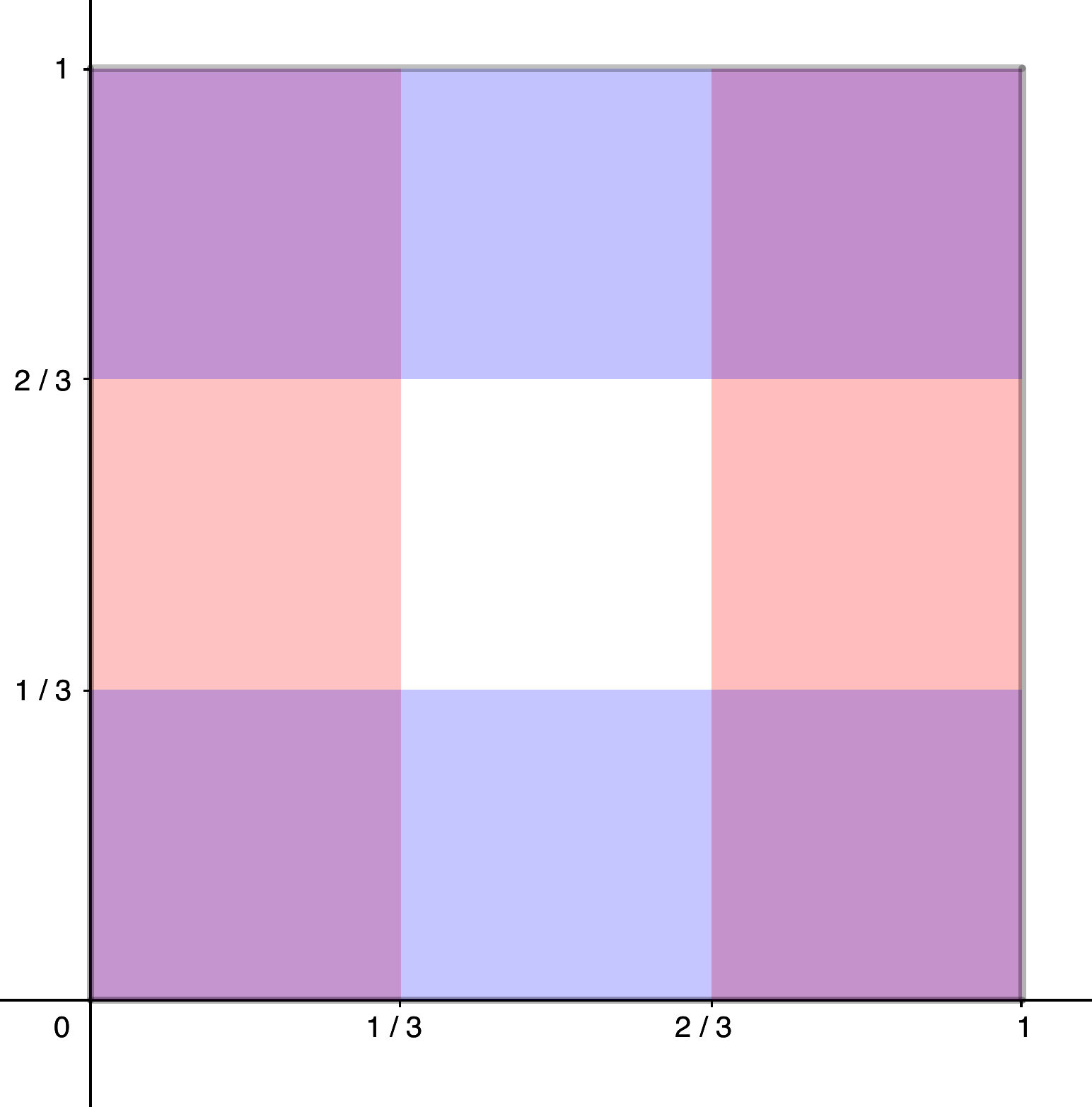}
\caption{An open baker's map.}
\end{subfigure}
\caption{Examples when $J=1$. The departure sets are in blue, the arrival sets in red. In the first example, $U_1 = ]-1,1[^2 \subset T^*\R$ , $D_{11} = ]-1,1[ \times ]-1/2,1/2[$, $A_{11}=]-1/2,1/2[ \times ]-1,1[$ with $F(x,\xi) = (x/2, 2 \xi)$. The trapped set is reduced to a single hyperbolic fixed point. The second example is built on the model of an open baker's map. We have $U_1 = ]0,1[^2 \subset T^*\R$ , $D_{11} = ]0,1[ \times ]0,1/3[ \cup ]0,1[ \times ]2/3,1[ $, $A_{11}=]0,1/3[ \times ]0,1[ \cup ]2/3,1[ \times ]0,1[$. In such a model, the map $F$ is piecewise affine and given by $F(x+a ,\xi)=(3x, a + \xi /3)$ for $a \in \{0 ,2 \}$, $(x,\xi) \in ]0,1[^2$. }
\label{figure_example_2}
\end{figure}

Here ends the description of the classical map. See Figure \ref{figure_example_2} for simple examples of such open hyperbolic maps. We then associate to $F$ \textit{open quantum hyperbolic maps}, which are its quantum counterparts. 

\begin{defi}\label{def_FIO}
Fix $\delta \in [0,1/2[$. 
We say that $T=T(h)$ is an \emph{open quantum hyperbolic map} associated with $F$, and we note $T=T(h) \in I_{\delta} ( Y \times Y , \Gr(F)^\prime )$ if : 
for each couple $(i,j) \in \{ 1, \dots , J \}^2$, there exists a semiclassical Fourier integral operator  $ T_{ i j }=T_{ i j } (h) \in I_{\delta} ( Y_j \times Y_i , \Gr(F_{ij})^\prime )$ associated with $F_{ i j}$ in the sense of definition \ref{Def_FIO_local}, such that
$$ T = ( T_{i j } )_{ 1 \leq i,j \leq J }  : \bigoplus_{i=1}^J L^2 (Y_i) \to \bigoplus_{i=1}^J L^2 (Y_i)$$ 
In particular $\WF^\prime (T) \subset \widetilde{A} \times \widetilde{D}$. 
We note $I_{0^+}(Y \times Y , \Gr(F)^\prime)= \bigcap_{ \delta >0}  I_{\delta} ( Y \times Y , \Gr(F)^\prime )$. 
\end{defi}

We will say that $T \in I_{0^+}(Y \times Y, \Gr(F)^\prime)$ is microlocally invertible near $\mathcal{T}$ if there exists a neighborhood $U^\prime \subset U$ of $\mathcal{T}$ and an operator $T^\prime \in I_{0^+}(Y \times Y , \Gr(F^{-1})^\prime)$  such that, for every $u =(u_1, \dots, u_J) \in L^2(Y) $  
$$ \forall j \in \{1, \dots, J \}, \WF(u_j) \subset U^\prime \cap U_j \implies TT^\prime u = u + \hinf ||u||_{L^2} , T^\prime T u = u +  \hinf ||u||_{L^2} $$

 Suppose that $T$ is microlocally invertible near $\mathcal{T}$ and recall that $T^* T \in \Psi_{0^+}(Y) \subset \Psi_{1/4}(Y)$ (this choice $\delta=1/4$ is arbitrary). Then, we can write
 $$T^* T = \op( a_h) + O(h^{1/4})_{L^2 \to L^2}$$ where $a_h$ is a smooth principal symbol in the class $S_{0^+}(U)$ (the definition of this symbol class is recalled in the appendix). We note $\alpha_h= \sqrt{|a_h|}$ and call it the \emph{amplitude} of $T$. Since $T$ is microlocally invertible near $\mathcal{T}$, $|a_h| >c^2$ near $\mathcal{T}$, for some $h$-independent constant $c>0$, showing that $\alpha_h$ is smooth and larger than $c$ in a neighborhood of $\mathcal{T}$. 
 
 \begin{rem}
 If $T$ has amplitude $\alpha$, at first approximation, $T$ transforms a wave packet $u_{\rho_0}$ of norm 1 centered at a point $\rho_0 $ lying in a small neighborhood of $\mathcal{T}$ into a wave packet of norm $\alpha(\rho_0)$ centered at the point $F(\rho_0)$. 
 \end{rem}

\subsubsection{Crucial resolvent bound.}
We now consider $M(h)$ an open quantum hyperbolic map, associated with $F$. We suppose that $M(h)$ is microlocally invertible near $\mathcal{T}$. Additionally, we make the following assumption :  there exists $L >0$ and $\phi_0 \in \cinfc(T^*Y,[0,1])$ such that $\supp (\phi_0)$ is contained in a compact neighborhood $\mathcal{W}$ of $\mathcal{T}$, $\mathcal{W} \subset \widetilde{D}$, $\phi_0 = 1$ in a neighborhood of $\mathcal{T}$ and  
\begin{equation}\label{small_outside_trapped_set}
M(h)(1 - \op(\phi_0) ) = O(h^{L})
\end{equation} 
Let us note $\alpha_h $ the amplitude of $M(h)$ and $||\alpha_h||_\infty$ its sup norm in $\mathcal{W}$. It is a priori $h$-dependent, but it is uniformly bounded in $h$. Proposition 4.1 in \cite{Vacossin} then states that : 

\begin{prop}\label{Prop_Crucial_Estimate} Suppose that $M(h)$ satisfies the above assumptions. 
There exists $\delta >0$, $ \gamma >0$, $h_0 >0$ and a family of integer $N(h) \sim \delta |\log h|$, defined for $0< h \leq h_0$, such that for all $0 < h \leq h_0$, 
\begin{equation}\label{equation_spectral_gap}
||M(h)^{N(h)}||_{L^2 \to L^2} \leq h^\gamma ||\alpha_h||_\infty^{N(h)} 
\end{equation}
\end{prop}

\begin{rem}\text{} 
\begin{itemize}
\item Strictly speaking, the result of \cite{Vacossin} applies to operators of the form $T(h)\op(\alpha)$ where $T$ is microlocally unitary near $\mathcal{T}$. We can reduce (\ref{equation_spectral_gap}) to this case. Indeed, locally near every point $\rho_0 \in \mathcal{T}$, $M(h)$ takes this form, and $\mathcal{T}$ is totally disconnected, so that $M(h)$ takes this form in a small neighborhood of $\mathcal{T}$. Finally, as showed in \cite{Vacossin} (Subsection 4.2), the behavior of $M(h)$  outside any neighborhood of $\mathcal{T}$ contributes as a $\hinf$ in (\ref{equation_spectral_gap}), as soon as $N(h)$ is bigger than a fixed $N_0$ depending on this neighborhood. 
\item Note that the constant $\delta$ and $\gamma$ are purely dynamical, that is, depend only on the dynamics of $F$ near $\mathcal{T}$. Indeed, $\delta$ is defined in Section 4.1 in \cite{Vacossin} using only dynamical parameters, such as the Jacobian of $F$. Concerning $\gamma$, it is implicitly defined using the porosity of the trapped set (see Section 6 in \cite{Vacossin}). $h_0$ depends on $\alpha$ (through a finite number of semi-norms). This remark will turn out to be important when dealing with scattering by a potential. 
\end{itemize}
\end{rem}

This estimate, which is the crucial point in \cite{Vacossin} to prove the spectral gap naturally leads to a resolvent bound for $(\Id - M(h))^{-1}$ : 

\begin{prop}\label{Prop_Estimate_M}
Suppose that $M(h)$ satisfies the above assumptions. 
Let $\gamma$ and $\delta$ be given in Proposition \ref{Prop_Crucial_Estimate} and assume that for some $h_1 >0$, for all $0 < h \leq h_1$, 
\begin{equation} 
||\alpha_h||_\infty < \exp\left( \frac{\gamma}{\delta} \right) 
\end{equation}
Let us consider $A \geq 1$ such that  for all $0 < h \leq h_1$,   $|| \alpha_h||_\infty \leq A$. Then, there exists $h_0 \in ]0,h_1]$ such that for all $0 < h \leq h_0$, 
\begin{equation}
\left|  \left| \left( \Id - M(h) \right)^{-1} \right| \right|_{L^2 \to L^2} \leq 
 2 \delta | \log h |  h^{- \delta \log A}
\end{equation}
\end{prop}

\begin{proof}
First recall that $M^*M \in \Psi_{0^+}$ with $\sigma_0 (M^*M) = \alpha_h^2$ and $M = O(h^L)$ microlocally outside $\mathcal{W}$. Hence, we can estimate the operator norm of $M(h)$ (see \cite{ZW}, Theorem 13.13), 
$$||M(h)||_{L^2 \to L^2} \leq ||\alpha_h||_\infty + O(h^\eta)$$
where $\eta$ is any fixed number in $]0,1[$. \\  
Let $N(h) \sim \delta |\log h|$ be the family of integers given by Proposition \ref{Prop_Crucial_Estimate}. Without loss of generality, we may assume that $N(h) \leq \delta | \log h |$. We use the fact that $h^\gamma ||\alpha_h||^{N(h)} = o(1)$ when $h \to 0$ if $||\alpha_h||_\infty < e^{ \frac{\gamma}{\delta}}$. As a consequence, $ \Id - M^{N(h)}$ is invertible for $h$ small enough with 
\begin{equation}\label{eq1_prop2.2}
 || (\Id - M^{N(h)} )^{-1} || \leq \frac{3}{2} \;  , \;  0 \leq h \ll 1
\end{equation}
This implies that $I - M$ is invertible with inverse 
\begin{equation}\label{eq2_prop2.2}
(\Id - M)^{-1} = (\Id + M + \dots + M^{N(h)-1}) (\Id - M^{N(h)})^{-1} 
\end{equation} 
We hence estimate
\begin{align*}
|| \Id + M + \dots + M^{N(h)-1} ||  \leq N(h) (||\alpha_h||_\infty + O(h^\eta)) ^{N(h)}
& \leq N(h) (A + O(h^\eta)))^{N(h)} \\
& \leq  \delta| \log h | h^{-\delta \log A}  (1 + o(1) )\\
& \leq \frac{4}{3} \delta |\log h | h^{-\delta \log A}
\end{align*}
 if $h$ is small enough. Using (\ref{eq2_prop2.2}), we multiply with (\ref{eq1_prop2.2}) and find the required inequality. \\
\end{proof}

\begin{rem}
The constant 2 can be changed into any $1 + \varepsilon$ by changing $h_0$ into $h_0(\varepsilon)$. \\
If  $\liminf_{h \to 0} ||\alpha_h||_\infty >1$ we can get rid of the $\log h$ term by changing it into a constant depending on $||\alpha_h||_\infty$. More precisely, a better estimate of the sum can show that 
$$ \left| \left| \left( \Id - M(h) \right)^{-1} \right| \right|_{L^2 \to L^2} \leq \frac{2}{||\alpha_h||_\infty -1 }h^{-\delta \log ||\alpha_h||_\infty} $$
The main interest of the estimate in Proposition \ref{Prop_Estimate_M} is that it gives a uniform estimate in the limit $||\alpha_h||_\infty \to 1$. 
\end{rem}

\subsection{Proof in the case of obstacle scattering. }

In this subsection, we recall the main ingredients of \cite{NSZ14} and prove the resolvent estimate of Theorem \ref{Thm2} in obstacle scattering.\footnote{We use notations similar to the ones in \cite{NSZ14} but beware that we do not use the exact same conventions. }

Let $\mathcal{O} = \bigcup_{i=1}^J \mathcal{O}_j$ where $\mathcal{O}_j$ are open, strictly convex obstacles in $\R^2$ having smooth boundary and satisfying the \emph{Ikawa's no-eclipse condition} : for $i \neq j \neq k$, $\overline{\mathcal{O}_i}$ does not intersect the convex hull of $ \overline{ \mathcal{O}_j }\cup \overline{\mathcal{O}_k}$. Let $P(h) = -h^2 \Delta_\Omega -1$ and fix a cut-off function $\chi \in \cinfc(\R^2)$ equal to one in a neighborhood of $\overline{\mathcal{O}}$. First note that by a simple scaling argument, it is enough to prove (\ref{eq_thm_2}) for $z \in \{ z \in D(0,Kh), \im z \geq -\gamma h \}$ for any $K>0$ fixed. 

\paragraph{Complex scaling. } We fix $R_\chi >0$ such that $\supp \chi \subset B(0,R_{\chi})$. For a parameter $\theta \in ]0, \pi/2[$, we consider a complex deformation $\Gamma_\theta \subset \C^2$ of $\R^2$ such that for some $R^\prime > R_\chi$, 
\begin{align*}
\Gamma_\theta \cap B_{\C^2}(0, R_\chi) = \R^2 \cap B_{\R^2}(0, R_\chi)  \\
\Gamma_\theta \cap \C^2 \setminus B_{\C^2}(0,R^\prime) = e^{i \theta} \R^2 \cap \C^2 \setminus B_{\C^2}(0,R^\prime) \\
\Gamma_\theta = f_\theta (\R^2) \; ; \; f_\theta : \R^2 \to \C^2 \text{ injective}
\end{align*}

By identifying $\R^2$ and $\Gamma_\theta$ through $f_\theta$, we note $\Delta_\theta$ the corresponding complex-scaled free Laplacian, and $\Delta_{\Omega,\theta}$ the complex scaled Laplacian on $H^2(\Omega) \cap H_0^1 (\Omega)$. We fix $K>0$ (which can be chosen arbitrarily large) and for $z \in D(0,Kh)$, we note 
\begin{equation}  
P_{\bullet}(z) =- h^2 \Delta_\bullet - 1 - z
\end{equation} 
with either $\bullet= \theta$ or $\bullet= \Omega,\theta$. We note the associated resolvent, when they are defined, $$ R_{\Omega,\theta}(z) : L^2 (\Omega) \to H^2(\Omega) \; ; R_{\theta}(z) : L^2(\R^2) \to H^2(\R^2)$$

\begin{rem}
With these notations, the parameter $\lambda$ of the usual resolvent $R(\lambda)$ takes the form $\lambda=\lambda_h(z)= h^{-1}(1+z)^{1/2}$ with $z \in D(0,Kh) \subset D(0,1)$  if $h$ small enough, so that the square root is well defined and gives a holomorphic function of $z$. 
\end{rem}
Thanks to the usual properties of the complex scaling method (see for instance \cite{DyZw}, Section 4.5 in Chapter 4 and the references given there), we have :
\begin{itemize}
\item The operators $P_\theta(z)$ and $P_{\Omega, \theta}(z)$ are Fredholm operators of index 0;
\item $z$ is a pole of $R_{\Omega,\theta}(z)$ if and only if $\lambda_h(z)$ is a scattering resonance ; 
\item For $z$ not a pole of $R_{\Omega,\theta}(z)$, in virtue of the properties of $\chi$ and $\Gamma_\theta$, we have (recall the definitions of $R(\lambda)$ and $R_h(z)$ in (\ref{resolvant}) and (\ref{semiclassical_resolvent}) respectively),  
$$ \chi R_{\Omega,\theta}(z) \chi = \chi R_h(z) \chi = h^{-2} \chi R(\lambda_h(z)) \chi$$
\item Finally, we recall that we have the following standard estimate for $R_\theta(z)$ (see for instance \cite{DyZw}, Theorem 6.10) 
\begin{equation}
||R_{\theta}(z)||_{L^2(\R^2) \to H^2_h(\R^2)} \leq Ch^{-1} \; ; \; z \in D(0,Kh)
\end{equation}
In particular, it tells that $R_{\theta}$ is holomorphic in $D(0,Kh)$. Here, $H_h^2(\R^2)$ is a semiclassical Sobolev space i.e. $H^2(\R^2)$ with the norm $||u||_{H_h^2(\R^2)} = || (1-h^2 \Delta)u||_{L^2}$. 
\end{itemize}

To prove Theorem \ref{Thm2}, it is then enough to give a bound for $\chi R_{\Omega,\theta}(z) \chi: L^2(\Omega) \to L^2(\Omega)$ in the corresponding region. 

\paragraph{Reduction to the boundary of the obstacles.} Following \cite{NSZ14} (Section 6), we introduce the following operators to obtain a reduction to the boundary. For $j= 1, \dots, J$, let 
\begin{equation}
\gamma_j : u \in H^2(\Omega) \mapsto u|_{\partial \mathcal{O}_j} \in  H^{3/2} (\partial \mathcal{O}_j)
\end{equation}
be the (bounded) trace operator and $\gamma u = (\gamma_j u)_j \in H^{3/2} (\partial \mathcal{O}) \coloneqq H^{3/2}(\partial \mathcal{O}_1)\times \dots \times H^{3/2}(\partial \mathcal{O}_J)$,
and let  
\begin{equation} H_j(z) : H^{3/2}(\partial \mathcal{O}_j) \to H^2(\R^2 \setminus \mathcal{O}_j ) \overset{\text{extension by 0} }{\longrightarrow}L^2 (\R^2)
\end{equation}  be the Poisson operator, defined, for $v \in H^{3/2}(\partial \mathcal{O}_j)$, as the solution to the problem 
$$ \left\{ \begin{array}{l}
P_\theta(z) H_j(z) v =0 \text{ in } \R^2 \setminus \overline{\mathcal{O}_j} \\
\gamma_j H_j(z)v = v .
\end{array} \right. $$
$u= H_j(z)v$ is a solution of the problem $P_\theta(z) u = 0$ with outgoing properties. So as $P_\theta(z)$, $H_j(z)$ implicitly depends on $h$. 
For $\overrightarrow{v} =(v_1, \dots, v_J) \in H^{3/2}(\partial \mathcal{O})$, we set 
$$ H(z) \overrightarrow{v} = \sum_{j=1}^J H_j(z) v_j $$
 Let us define the following operator-valued matrix $\mathcal{M}(z) : H^{3/2} (\partial \mathcal{O}) \to H^{3/2} (\partial \mathcal{O}) $ by the relation 
\begin{equation}
\Id - \mathcal{M}(z) = \gamma H(z) 
\end{equation}

We state a few facts concerning these operators. In the following lemma, we give estimates involving the semiclassical version of the Sobolev spaces $H^2(\R^2 \setminus \overline{\mathcal{O}_j})$ and $H^{3/2}(\partial{O}_j)$, denoted $H^2_h(\R^2 \setminus \overline{\mathcal{O}_j})$ and $H^{3/2}_h(\partial{O}_j)$ respectively. 

\begin{lem}\label{Lemma_gamma_j}
For $j =1, \dots, J$, there exists $C>0$ such that for all $0<h\leq 1$, the norm of the bounded operator $\gamma_j$ from $H_h^{2}(\R^2 \setminus \overline{\mathcal{O}_j}) $ to $H^{3/2}_h (\partial \mathcal{O}_j)$ satisfies 
$$ ||\gamma_j||_{H_h^{2}(\R^2 \setminus \overline{\mathcal{O}_j}) \to H^{3/2}_h (\partial \mathcal{O}_j)} \leq C h^{-1/2} $$
\end{lem}

\begin{proof}Using a partition of unity argument and local charts, it is sufficient to prove that the above result holds with $\R^2 \setminus \overline{\mathcal{O}}_j$ replaced by $\R \times \R^*_+$ and $\partial \mathcal{O}_j$ replaced by $\R$. In this setting, we note $\gamma$ the associated trace operator. 
First, we extend an element $u \in H^2_h(\R \times \R^*_+)$ to an element $\tilde{u} \in H^2_h(\R \times \R)$ such that $||\tilde{u}||_{H_h^2} \leq C ||u||_{H^2_h}$ (see for instance \cite{Evans}, Chapter 5, Section 4 : in the proof of Theorem 1, one can extend $u \in H_h^{2}(\R \times \R_+^*)$ with the formula : for $y>0$, $u(x,-y)=-3u(x,y)+4 u(x,-y/2)$). Then we observe that, with $\mathcal{F}^1_h$ (resp. $\mathcal{F}_h$) the semiclassical unitary Fourier transform in 1D (resp. 2D), 
$$ \mathcal{F}^1_h (\gamma u)(\xi)  = \frac{1}{(2\pi h)^{1/2}} \int_\R \mathcal{F}_h( \tilde{u})(\xi, \eta) d \eta $$
From, this we get 
\begin{equation}\label{inequality_trace_CS}
||\gamma u ||_{H_h^{3/2}} \leq Ch^{-1/2} ||\tilde{u}||_{H_h^2}
\end{equation}
Indeed, by Cauchy-Schwarz, we have 
\begin{align*}
\left| \int_\R \mathcal{F}_h( \tilde{u})(\xi, \eta) d \eta \right|^2 &\leq \left( \int_\R |\mathcal{F}_h( \tilde{u})(\xi, \eta)|^2 (1 + \xi^2 + \eta^2)^{2} d \eta \right) \left( \int_\R  (1 + \xi^2 + \eta^2)^{-2}d \eta \right)  \\
& \leq \left( \int_\R |\mathcal{F}_h( \tilde{u})(\xi, \eta)|^2 (1 + \xi^2 + \eta^2)^{2} d \eta \right) (1+ \xi^2)^{-3/2} \int_\R (1+ \eta^2)^{-2} d \eta 
\end{align*}
We find (\ref{inequality_trace_CS}) by multiplying by $(1+ \xi^2)^{3/2}$ and integrating over $\xi$. 
This concludes the proof. 
\end{proof}

\begin{lem}\label{Lemma_H_j}
For $j =1, \dots, J$, for any $K>0$, there exists $h_0 >0$ such that for all $0 < h \leq h_0$, $\chi H_j(z)$ is holomorphic in $D(0,Kh)$ and satisfies for some $C>0$ independent of $h$, and for $z \in D(0,Kh)$, 
$$ ||\chi  H_j(z)  ||_{ H^{3/2}_h(\partial \mathcal{O}_j) \to L^2(\R^2 \setminus \overline{\mathcal{O}_j} )}  \leq Ch^{-1/2}$$
\end{lem}

\begin{proof}
We follow the main lines of the proof of Lemma 6.1 in \cite{NSZ14}. \\
First, let us introduce an extension operator $T_j^h : H_h^{3/2}(\partial \mathcal{O}_j) \to H_h^2 (\R^2)$ such that for $v \in H_h^{3/2}(\partial \mathcal{O}_j)$, $T_j^h v$ is supported in a small neighborhood of $\partial \mathcal{O}_j$ and $$T_j^h = O(h^{1/2}) :H_h^{3/2}(\partial \mathcal{O}_j) \to H_h^2 (\R^2)$$
This is possible, for instance by taking the extension operator given in the proof of Lemma 6.1 in \cite{NSZ14}. Another approach consists in using a partition of unity and local charts to replace $\partial \mathcal{O}_j$ by $\R$, as in the proof of Lemma \ref{Lemma_gamma_j}. Then, one can consider the following operator 
$$ T_h : v \in H_h^{3/2} (\R) \mapsto T_h v \in H_h^{2}(\R^2) \; ; \; T_hv (x,y) = \left( \chi \left( \frac{x}{h} \langle h D_y \rangle \right) v \right)(y)$$
where $\chi \in \cinfc(\R)$, $\chi(0)=1$. Then, $T_h v(0,y) = v$ and one has 
$$ ||T_h v||_{H_h^{2}(\R^2)} \leq C h^{1/2} ||v||_{H_h^{3/2}(\R)}$$
Indeed, one has 
$$ \mathcal{F}_h (T_h v) (\xi, \eta) =h^{1/2} \langle \eta \rangle^{-1} \mathcal{F}_{1}^1 (\chi) \left(\frac{\xi}{\langle \eta \rangle} \right) \times  \mathcal{F}_h^1 (v) (\eta) $$ 
and hence 
\begin{align*}
||T_h v ||_{H_h^{2}(\R^2)}^2  &= \int_{\R^2} |\mathcal{F}_h (T_h v) (\xi, \eta) |^2(1+ \xi^2 + \eta^2)^2 d \xi d \eta  \\
& \leq h \int \langle \eta \rangle^{-2} \left| \mathcal{F}_{1}^1 (\chi) \left(\frac{\xi}{\langle \eta \rangle} \right)\right|^2  \left| \mathcal{F}_h^1 (v) (\eta) \right|^2 (1+\xi^2)^2 (1+ \eta^2)^2 d \eta d\xi \\
&\leq h \left( \int_\R | \mathcal{F}_1^1 \chi (\xi)|^2 (1+ \xi^2)^2 d \xi \right) \left( \int_\R  |\mathcal{F}_h^1 (\eta)|^2 (1+ \eta^2)^{3/2} d \eta \right) \\
& \leq C h ||v||^2_{H_h^{3/2}(\R)}
\end{align*}
We then assume that for all $v \in H_h^{3/2}(\partial \mathcal{O}_j)$, $\supp (T_j^h v) \subset \supp \chi$. Then, we claim that
$$ H_j(z) = \mathds{1}_{\R^2 \setminus \overline{\mathcal{O}_j}} T_j^h - R_{j, \theta}(z)  \mathds{1}_{\R^2 \setminus \overline{\mathcal{O}_j}} P_\theta(z)T_j^h $$
where $R_{j,\theta}(z)$ is the resolvent of the complex scaled Dirichlet realization of $-h^2 \Delta - 1$ on $\R^2 \setminus \overline{\mathcal{O}_j}$. Indeed, the boundary condition on $\partial \mathcal{O}_j$ is satisfied since $\text{Ran}( R_{j,\theta}) \subset H_0^1(\R^2 \setminus \overline{\mathcal{O}_j})$, and by definition, $P_\theta R_{j,\theta} w = w $ in $\R^2 \setminus  \overline{\mathcal{O}_j}$, for $w \in L^2(\R^2 \setminus  \overline{\mathcal{O}_j})$. \\
As a consequence, is suffices to show that  $\chi R_{j,\theta}(z) \chi$ is holomorphic in $D(0,Kh)$ with the bound 
$$\chi R_{j,\theta}(z) \chi = O(h^{-1}) : L^2(\R^2 \setminus \overline{\mathcal{O}_j}) \to L^2 (\R^2 \setminus \overline{\mathcal{O}_j})$$
This is a rather standard non-trapping estimates (here, when there is a single obstacle, the billiard flow is non-trapping). As explained in \cite{NSZ14} in the proof of Lemma 6.1, such an estimate relies on propagation of singularities concerning the wave propagator : one can check that an abstract non-trapping condition for black box Hamiltonian is satisfied (see for instance \cite{DyZw}, Definition 4.42). This implies that the required statement (\cite{DyZw}, Theorem 4.43) holds. 
\end{proof}

Finally, we recall the crucial relation between $R_{\Omega,\theta}(z)$ and $\mathcal{M}(z)$ (see \cite{NSZ14}, formula 6.11 and the references given there). Assume that $z \in D(0,Kh)$ and that $\Id- \mathcal{M}(z)$ is invertible. Then, so is $R_{\Omega,\theta}(z)$ and we have

\begin{equation}\label{Key_formula_from_M_to_R}
R_{\Omega,\theta}(z) = \mathds{1}_\Omega R_{\theta}(z) - \mathds{1}_\Omega H(z) (\Id - \mathcal{M}(z) )^{-1} \gamma R_{\theta}(z)
\end{equation}
In particular, we see that if we have a bound for $$||(\Id- \mathcal{M}(z))^{-1}||_{H_h^{3/2} \to H_h^{3/2}}$$
 we find a resolvent bound for $R_{\Omega,\theta}(z)$. In fact, as explained in \cite{NSZ14}, it is sufficient to work on $L^2(\partial\mathcal{O})$ in virtue of the following result : 

\begin{lem}(\cite{NSZ14}, Lemma 6.5)
For $j = 1, \dots, J$, let $B^* \partial\mathcal{O}_j \coloneqq \{ (y,\eta), y \in \partial\mathcal{O}_j, |\eta| \leq 1 \}$ and consider $\chi_j \in \cinfc(T^* \partial\mathcal{O}_j)$ such that $\chi_j = 1$ near $B^*\partial\mathcal{O}_j$. Then, by denoting $\op(\chi_j)$ a quantization of $\chi_j$ and by $D$ the diagonal operator-valued matrix $\text{Diag}(\op(\chi_1) , \dots, \op(\chi_J))$, we have 
\begin{align*}
(\Id - D) (\Id - \mathcal{M}(z) ) = \hinf_{ L^2(\partial\mathcal{O}) \to \cinf(\partial\mathcal{O})} \\
(\Id - \mathcal{M}(z) )(\Id - D )= \hinf_{ L^2(\partial\mathcal{O}) \to \cinf(\partial\mathcal{O})}
\end{align*}
\end{lem}
\vspace*{0.5cm}
As a consequence of this lemma, $\Id- \mathcal{M}(z)$ extends to an operator $L^2(\partial\mathcal{O}) \to L^2(\partial\mathcal{O})$ and as soon as $\Id-\mathcal{M}(z)$ is invertible and $z \in D(0,Kh)$
$$ ||(\Id -\mathcal{M}(z))^{-1} ||_{H^{3/2}_h \to H^{3/2}_h} \leq C_1  ||(\Id -\mathcal{M}(z))^{-1} ||_{L^2 \to L^2}$$
(with a constant $C_1$ independent of $z$). 

\paragraph{Microlocal properties of $\mathcal{M}(z)$ and reduction to a simpler problem. }
We recall the main microlocal properties of $\mathcal{M}(z)$ and reduce the invertibility of $\Id - \mathcal{M}(z)$ to a nicer Fourier integral operator, as explained in \cite{NSZ14} (Section 6). To do so, let us introduce the following notations. 

For $j \in \{1, \dots, J \}$, let $B^* \partial \mathcal{O}_j$ be the co-ball bundle of $ \partial \mathcal{O}_j$, $S^*_{ \partial \mathcal{O}_j} $ be the restriction of $S^* \Omega$ to $ \partial \mathcal{O}_j$, $\pi_j : S^*_{ \partial \mathcal{O}_j} \to B^* \partial \mathcal{O}_j$ the natural projection and $\nu_j(x)$ be the outward normal vector at $x \in  \partial \mathcal{O}_j$ (see Figure \ref{figure_billiard_map_a}).

\begin{figure}
\begin{subfigure}{0.4 \textwidth}
\centering
\includegraphics[scale=0.3]{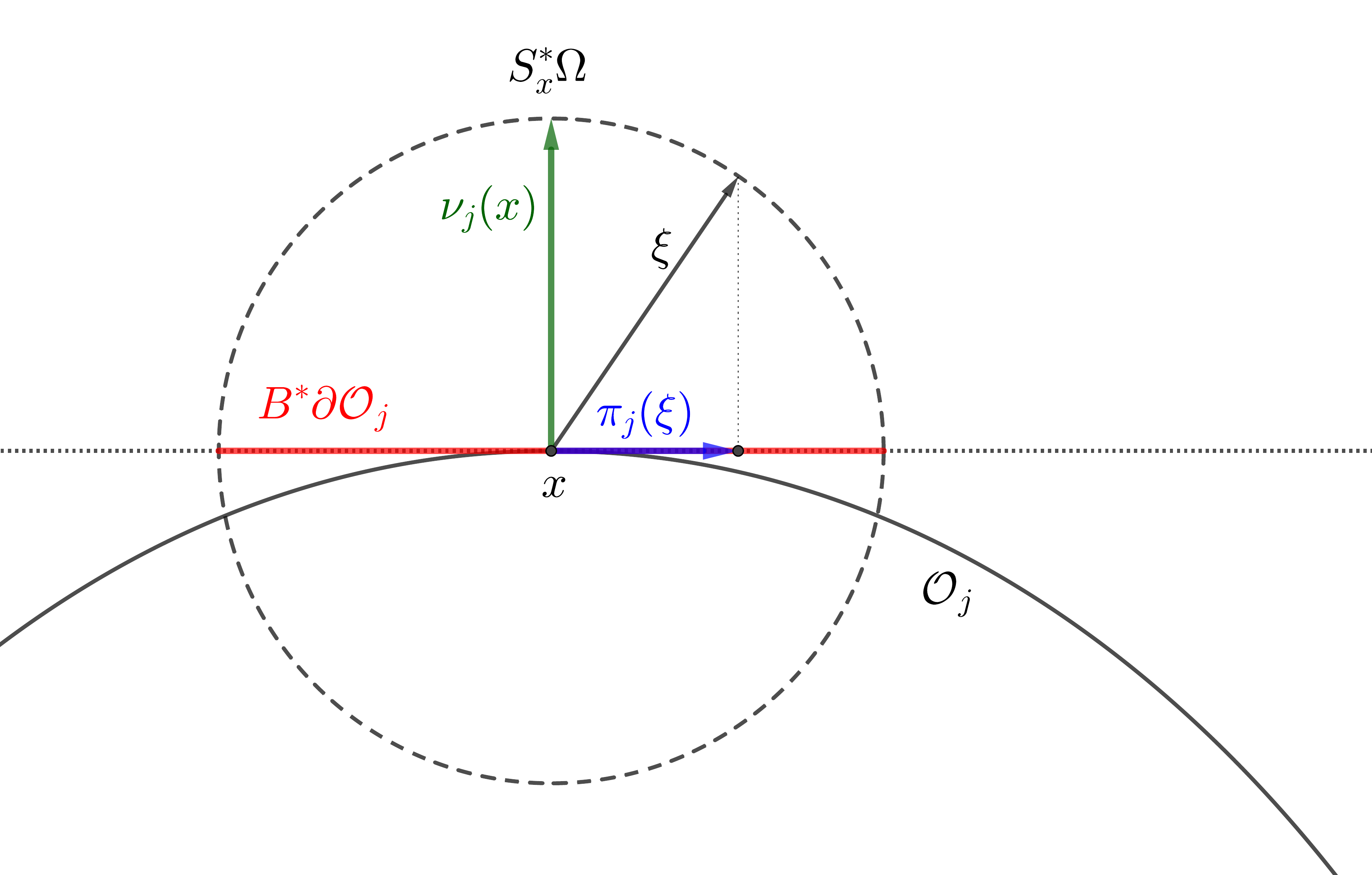}
\caption{The notations used to define the billiard map and the shadow map. }
\label{figure_billiard_map_a}
\end{subfigure} \hspace{0.3cm}\vline  \hspace{0.3cm}
\begin{subfigure}{0.4 \textwidth} 
\centering
\includegraphics[scale=0.33]{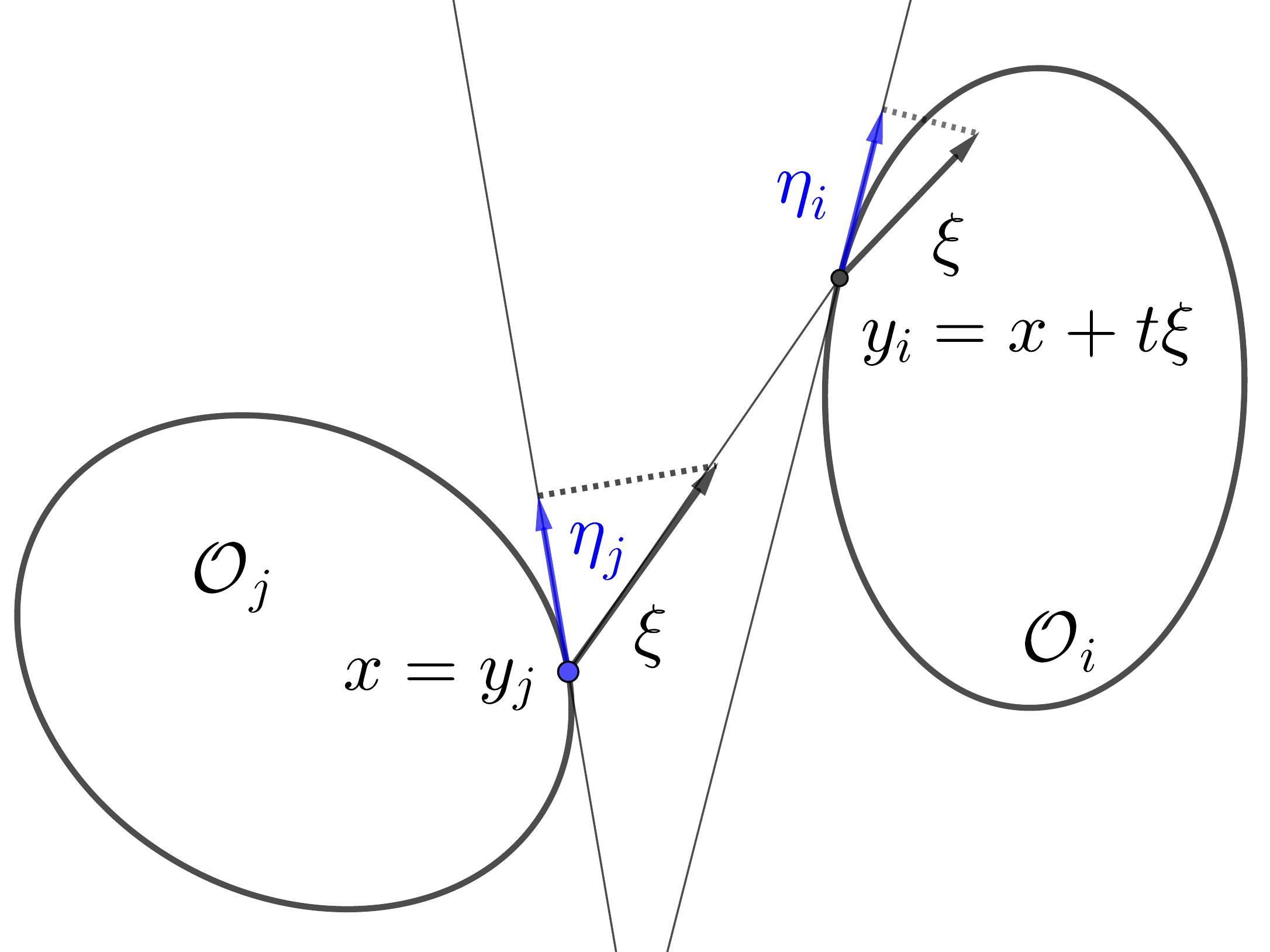}
\caption{The billiard map. $B^+_{ij}(y_j,\eta_j) = (y_i,\eta_i)$.  }
\label{figure_billiard_map_b}
\end{subfigure}
\hspace*{0.4cm}
\begin{subfigure}{0.5 \textwidth}
\centering
\includegraphics[scale=0.3]{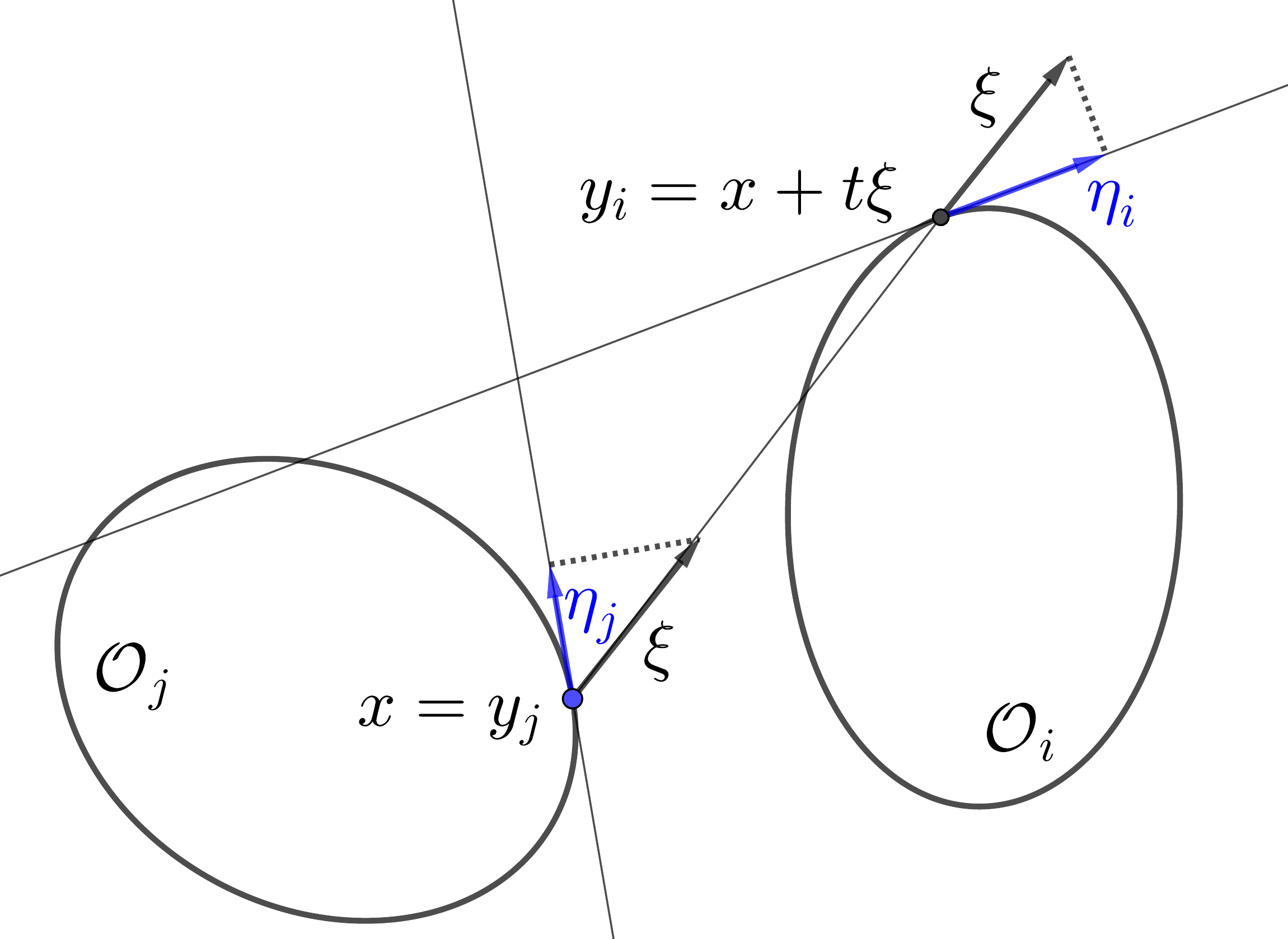}
\caption{The shadow map. $B^-_{ij}(y_j,\eta_j) = (y_i,\eta_i)$.}
\label{figure_billiard_map_c}
\end{subfigure} \hspace*{0.3cm}\vline  \hspace*{0.3cm}
\begin{subfigure}{0.4 \textwidth}
\centering
\includegraphics[scale=0.3]{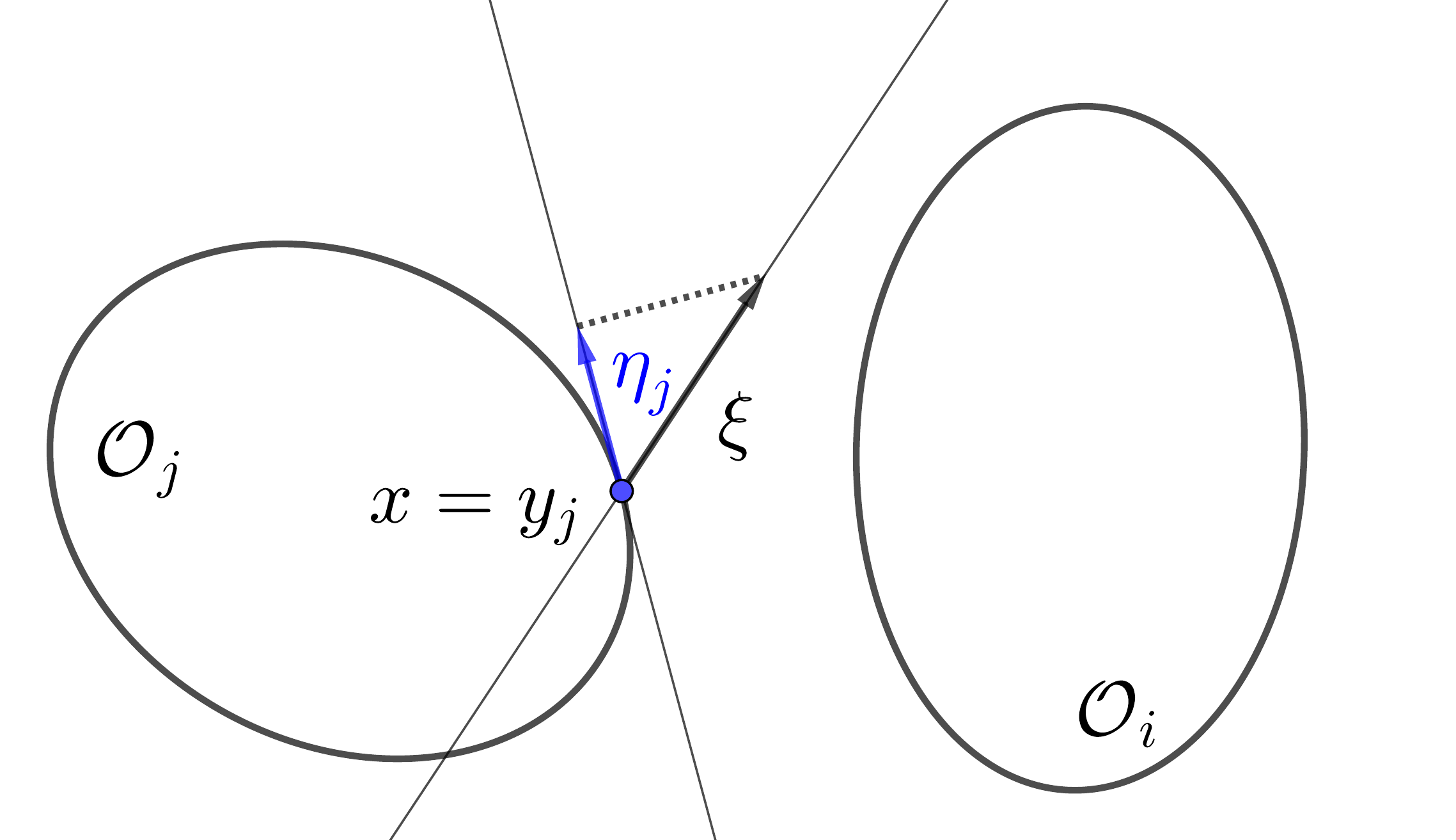}
\caption{These maps are open. In this figure, the point $(y_j,\eta_j)$ has no image.  }
\label{figure_billiard_map_d}
\end{subfigure}
\caption{Description of billiard map and the shadow map. }
\label{figure_billiard_map}
\end{figure}

 For $i \neq j$, let $\mathcal{B}_{ij}^\pm : B^* \partial \mathcal{O}_j  \to B^* \partial \mathcal{O}_i $ be the symplectic open maps defined by 
\begin{gather}\label{billiard_map_definition}
 \rho= \mathcal{B}^\pm_{ij}(\rho^\prime) \iff 
 \exists t >0 \, , \, \exists \xi \in \mathbb{S}^1 \, , \, \exists x \in  \partial \mathcal{O}_j \\
x + t\xi \in  \partial \mathcal{O}_i \, , \, \langle \nu_j(x),\xi \rangle >0 \, , \,  \pm \langle \nu_i(x+t\xi) \, , \,\xi\rangle <0, \\
\pi_j(x,\xi) = \rho^\prime \, , \, \pi_i (x+t\xi, \xi) = \rho
\end{gather}
$\mathcal{B}^+_{ij}$ is the billiard map, whereas $\mathcal{B}^-_{ij}$ is a shadow map (see Figure \ref{figure_billiard_map_b} and \ref{figure_billiard_map_c}).These maps are open. (see Figure \ref{figure_billiard_map_d}). Note that due to our definition of these maps, the glancing rays (that is the rays associated with a point $\rho =(x,\xi) \in B^*\partial \mathcal{O}$ with $|\eta|=1$) are not in the set of definition of $\mathcal{B}^\pm_{ij}$. Moreover, due to Ikawa's condition, if a point $\rho \in B^* \partial \mathcal{O}_j$ has an image by $\mathcal{B}_{ij}^\pm$, it cannot have one by $\mathcal{B}_{kj}^\pm$ for $k \neq i$. 
Let $\overline{A}_{ij}$ be the closure of the arrival set of the billiard map, that is 
$$ \overline{A}_{ij} = \overline{ \{ \rho \in B^* \partial \mathcal{O}_i, \exists \rho^\prime \in B^* \partial \mathcal{O}_j, \rho=\mathcal{B}_{ij}^+ (\rho^\prime)  \}}$$
Similarly, let $\overline{D}_{ij}$ be the closure of the departure set of the billiard map, that is 
$$ \overline{D}_{ij} = \overline{ \{ \rho^\prime \in B^* \partial \mathcal{O}_j, \exists \rho \in B^* \partial \mathcal{O}_i,  \rho=\mathcal{B}_{ij}^+ (\rho^\prime)\}}$$
We also note $$\overline{A}_i = \bigsqcup_{j \neq i} \overline{A}_{ij} \; ; \; \overline{D}_i = \bigsqcup_{j \neq i} \overline{D}_{ji}$$
Finally, we introduce the arrival and departure glancing regions : 
$$ \widetilde{A}_i^{\mathcal{G}} = \overline{A}_i \cap S^* \partial \mathcal{O}_i \; ; \;  \widetilde{D}_i^{\mathcal{G}} = \overline{D}_i \cap S^* \partial \mathcal{O}_i  $$
We recall the main facts proved in \cite{NSZ14} concerning these relations and their link with $\mathcal{M}(z)$ : 

\begin{lem}\label{Lem_no_eclipse}
(See Figure \ref{figure_csq_no_eclipse}). Assuming that the obstacles satisfy the no-eclipse condition, the following holds : let $i \neq j \neq k$. Let $( \rho_1, \rho_1^\prime) \in \overline{\Gr(\mathcal{B}_{ji}^-) }$ and $(\rho_2,\rho_2^\prime)  \in \overline{\Gr(\mathcal{B}_{kj}^\pm)}$. Then, 
$$\rho_1 \neq \rho_2^\prime$$
\end{lem}

\begin{figure}[h]
\includegraphics[width=10cm]{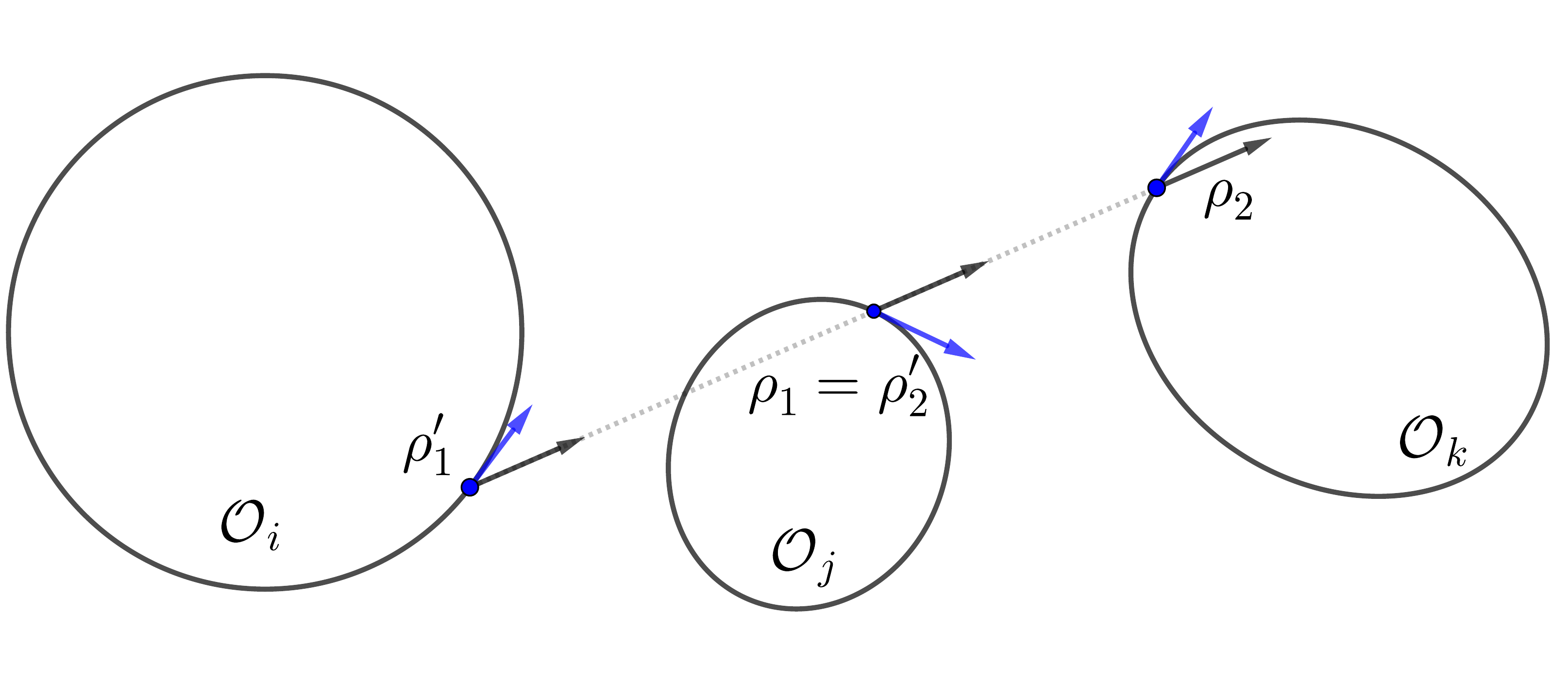}
\caption{The no-eclipse condition prevents such situation. The points $\rho_1,\rho_2,\rho_1^\prime $ are represented on the form $\rho=(y,\eta)$ with a blue point for $y \in \partial \mathcal{O}$ and a blue arrow for $\eta \in B^*_y \partial \mathcal{O}$. The limit situations, where the dotted line would be tangent to one of the obstacle are also excluded. } 
\label{figure_csq_no_eclipse}
\end{figure}

In particular, it is possible to consider open neighborhoods $U_i^A$ and $U_i^D$ of $\widetilde{A}_i^{\mathcal{G}}$ and $\widetilde{D}_i^{\mathcal{G}}$ respectively, such that (see Figure \ref{figure_glancing}), by noting $\pi_R$ (resp. $\pi_L$) the projection $(\rho^\prime, \rho) \mapsto \rho$ (resp. $(\rho^\prime, \rho) \mapsto \rho^\prime$) 

\begin{align*}
\rho^\prime \in U_i^A \implies  \rho^\prime \not \in \bigcup_{k \neq i} \pi_R \left( \Gr( \mathcal{B}_{ki}^\pm) \right) \text{ (see Figure \ref{subfigure_A} : it means that the ray continues to infinity)}  \\
\rho \in U_i^D \implies\rho \not \in \bigcup_{k \neq i} \pi_L \left( \Gr( \mathcal{B}_{ik}^\pm) \right)  \text{ (see Figure \ref{subfigure_B} : it means that the ray comes from infinity)}
\end{align*}

\begin{figure}[h]
\begin{subfigure}{0.48 \textwidth}
\centering
\includegraphics[width=7cm]{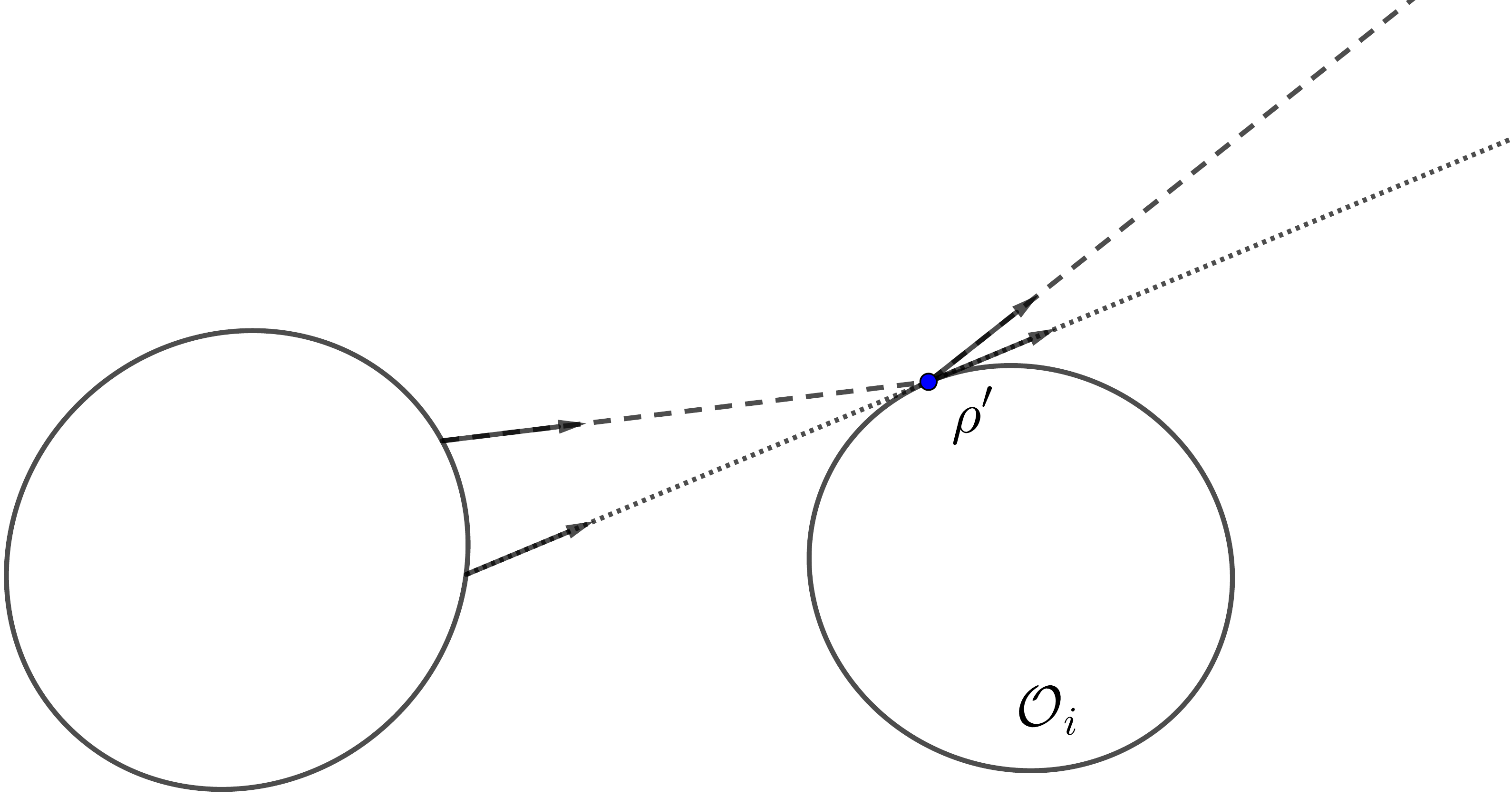}
\caption{Due to the no-eclipse condition, if a trajectory coming from an obstacle becomes glancing, then it goes on to infinity without hitting another obstacle. This holds in a neighborhood of the glancing ray and allows to define $U_i^A$.}
\label{subfigure_A}
\end{subfigure}
\vrule
\hspace*{0.1cm}
\begin{subfigure}{0.48 \textwidth}
\centering
\includegraphics[width=7cm]{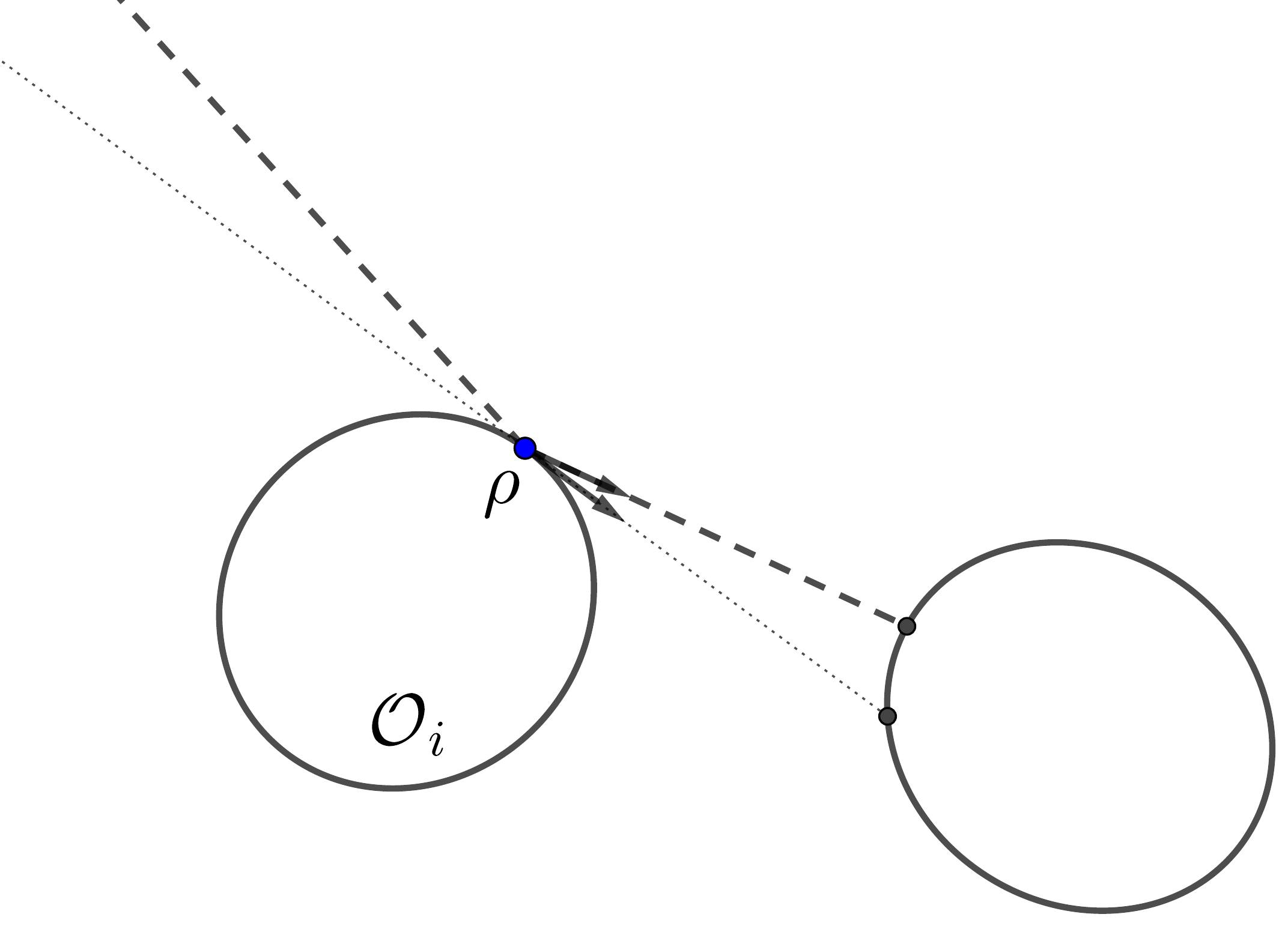}
\caption{Due to the no-eclipse condition, if a glancing trajectory hits an obstacle, it means that the ray comes from infinity. This holds in a neighborhood of the glancing ray and allows to define $U_i^D$.}
\label{subfigure_B}
\end{subfigure}
\caption{The sets $U_i^A$ and $U_i^D$ are built by using the properties of the glancing rays. The dotted lines correspond to glancing rays, the broken lines represent trajectories close to the glancing region.  }
\label{figure_glancing}
\end{figure}

Let us fix cut-off functions $\chi_i^{D}$ (resp. $\chi_i^{A}$) such that $\chi_i^{D}=1$ near $D_i^{\mathcal{G}}$ (resp. $\chi_i^A =1$ near $A_i^{\mathcal{G}}$) and $\supp \chi_i^{D} \subset U_i^{D}$ (resp. $\supp \chi_i^{A} \subset U_i^{A}$). 

We gather the results of Proposition 6.7 in \cite{NSZ14} and some of its consequence in the following proposition. It is based on the microlocal analysis of the operators involved. 
\begin{prop}\label{prop_quantum_billiard}
For $i \neq j$ , 
\begin{itemize}
\item uniformly in $z \in D(0,Kh)$,  $$  \mathcal{M}_{ij}(z) \op(\chi_j^A) =  \hinf_{L^2 \to \cinf}  \; ; \;  \op(\chi_i^D ) \mathcal{M}_{ij}(z) = \hinf_{L^2 \to \cinf} $$ 
\item By excluding the glancing region on the left and on the right, we have 
$$(1 - \op(\chi_i^D)) \mathcal{M}_{ij}(z) (1- \op(\chi^A_j))  \in I_0 (\partial \mathcal{O}_i \times \partial \mathcal{O}_j, \Gr\left( \mathcal{B}_{ij}^+\right)^\prime ) +I_0 (\partial \mathcal{O}_i \times \partial \mathcal{O}_j, \Gr(\left( \mathcal{B}_{ij}^-\right)^\prime )$$
so let us write
$$ (1 - \op(\chi_i^D)) \mathcal{M}_{ij}(z) (1- \op(\chi^A_j)) = \mathcal{M}_{ij}^+ (z) + \mathcal{M}_{ij}^-(z)$$
with $\mathcal{M}^\pm_{ij}(z) \in I_0 (\partial \mathcal{O}_i \times \partial \mathcal{O}_j, \Gr\left( \mathcal{B}_{ij}^\pm\right)^\prime ) $. Only compact parts of the interior of the graphs of $\mathcal{B}_{ij}^\pm$ are involved in the definition of the class  $I_0 (\partial \mathcal{O}_i \times \partial \mathcal{O}_j, \Gr\left( \mathcal{B}_{ij}^\pm\right)^\prime )$, depending on the support of $\chi_i^D$ and $\chi_j^A$ (see \ref{appendix_symplectic_map_quantization} in the appendix, for a description of this class). 
\item The operators $\mathcal{M}_{ij}^\pm(z)$ have amplitude $\alpha^\pm_{ij}(z)$ satisfying, for $z \in D(0,Kh)$ and for some $C_1, \tau >0$, 
$$ \alpha^\pm_{ij} (z) \leq C_1 e^{ - \tau \frac{\im z}{h}} $$ 
\item  Finally, in virtue of Lemma \ref{Lem_no_eclipse}, $\mathcal{M}_{ij}^\pm(z) \circ \mathcal{M}_{jk}^-(z) = \hinf_{L^2 \to \cinf}$ uniformly for $z \in D(0,Kh)$. 
\end{itemize}
\end{prop}
Let us note $\mathcal{M}^\pm(z)$ the matrix of operators with 
$$  \left( \mathcal{M}^\pm(z)\right)_{ij} =\left\{ \begin{array}{l}
\mathcal{M}^\pm_{ij}(z) \text{ if } i \neq j \\
0 \text{ if } i=j
\end{array} \right.$$
Then, we observe that 

\begin{equation}
\left( \Id- \mathcal{M}(z) \right)(\Id + \mathcal{M}^-(z)) = \Id- \mathcal{M}^+(z) + \hinf_{L^2 \to \cinf}
\end{equation}
Since we are interested in invertibility in strips, let's note :
 $$ \Omega(\gamma,K, h) = D(0,Kh) \cap \{ \im z \geq - \gamma h\} $$  We have the rather obvious lemma : 

\begin{lem}\label{Lemma_bound_mathcalM}
Assume that for $z \in \Omega(\gamma,K,h)$, $\Id - \mathcal{M}^+(z)$ is invertible and satisfies the bound 
$$ ||(\Id- \mathcal{M}^+(z))^{-1}||_{L^2 \to L^2} \leq a(z,h)$$
with $a(z,h) \leq h^{-N}$ for some $N$ independent of $z \in \Omega(\gamma,K,h)$. Then, there exists $h_0 >0$ and $C>0$, such that
for $0< h \leq h_0$, and for all $z \in \Omega(\gamma,K,h)$, $\Id - \mathcal{M}(z)$ is invertible and satisfies
$$ ||(\Id - \mathcal{M}(z))^{-1}||_{L^2 \to L^2} \leq C a(z,h)$$
\end{lem}

\begin{proof}
Assuming the invertibility of $\Id - \mathcal{M}^+(z)$, it suffices to write 
$$\left( \Id - \mathcal{M}(z)\right) (\Id + \mathcal{M}^-(z)) (\Id - \mathcal{M}^+(z))^{-1} = \Id + R(z,h)$$
with $R(z,h) = (\Id - \mathcal{M}^+(z))^{-1} \hinf_{L^2 \to \cinf} = \hinf_{L^2 \to L^2}$ uniformly in $z$. We conclude by a Neumann series argument to invert the right hand side $\Id + R(z,h)$ and use the bound on the amplitude of $\mathcal{M}^-(z)$ given in Proposition \ref{prop_quantum_billiard}, which gives a uniform bound for $\mathcal{M}^-(z)$ in $D(0,Kh)$. 
\end{proof}

It is then enough to prove the invertibility of $\Id -\mathcal{M}^+(z)$ with polynomial resolvent bounds, where $\mathcal{M}^+(z)$ is associated with the billiard map.

\paragraph{Conjugation by an escape function. }
The operator $\mathcal{M}^+(z)$ satisfies almost all the assumptions of Proposition \ref{Prop_Crucial_Estimate} for the relation $F = \mathcal{B}^+$, except that it is not very small outside a fixed compact neighborhood of $\mathcal{T}$ \footnote{strictly speaking, $\partial \mathcal{O}$ is a not a disjoint union of intervals, but since we work with the relation $\mathcal{B}^+$, we can use microlocal cut-offs to restrict to the relevant part of the obstacles, which is included in a disjoint union of open intervals}. To fix this problem, following \cite{NSZ14} (Section 6.3), we can introduce a smooth escape function $g_0$. Recall that $\mathcal{T}$ is the trapped set for $F = \mathcal{B}^+$ and let $\mathcal{W}_1 \Subset \mathcal{W}_2 \Subset \mathcal{W}_3$ be  subsets of $\bigcup_i \widetilde{D}_i$ such that $\mathcal{T} \subset \mathcal{W}_1$ and such that $\mathcal{W}_3$ is large enough so that $\mathcal{M}^+(z) \op(\phi) = \hinf_{L^2 \to \cinf}$ for any smooth function $\phi$ such that $\supp (\phi) \cap \mathcal{W}_3 = \emptyset$. This is possible in virtue of the third point in Proposition \ref{prop_quantum_billiard}. Concerning $\mathcal{W}_2$, it can be an arbitrarily small neighborhood of $\mathcal{T}$. Then, one can construct $g_0$ such that (see Lemma 4.5 in \cite{NSZ14}), 
\begin{align*}
&g_0 = 0 \text{ in } \mathcal{W}_1 \\
&g_0 \circ F - g_0 \geq 0 \text{ in } \mathcal{W}_3 \\
&g_0 \circ F - g_0 \geq 1 \text{ in } \mathcal{W}_3 \setminus \mathcal{W}_2
\end{align*}

Then, we set $g = T \log(1/h) g_0$ for some $T >0$ fixed and large enough, so that $ e^{\pm \op(g)}$ are pseudodifferential operators and satisfies 
$$ e^{\pm \op(g)} = O\left( h^{-TC_0} \right) \; ; \; C_0 = ||g_0||_\infty$$
(Note that $\op(g)$ is a diagonal matrix-valued operator on $L^2(\partial \mathcal{O}) = \bigoplus L^2 (\partial \mathcal{O}_i)$), and in virtue of Egorov's theorem, the operator 
$$\mathcal{M}_g^+ \coloneqq e^{-\op(g)} \mathcal{M}^+(z) e^{\op g} $$
is $O(h^L)$ for some $L>0$, microlocally outside a neighborhood of $\mathcal{T}$, which can be made as small as necessary if $\mathcal{W}_2$ is small enough. 

\paragraph{End of proof. }
We can now apply Proposition \ref{Prop_Crucial_Estimate} and then, Proposition \ref{Prop_Estimate_M}, to $M(h)=\mathcal{M}_g^+(z)$ for $z \in D(0,Kh)$ with $ \im z \geq - \gamma h$. To control the amplitude $\alpha_h(z)$ of $\mathcal{M}_g^+(z)$, we simply need a bound in a small neighborhood of $\mathcal{T}$ in which $\mathcal{M}_g^+$ is not $O(h^L)$. In virtue of Egorov's theorem, the amplitude of $\mathcal{M}_g^+$ is smaller than the amplitude of $\mathcal{M}^+$. We now claim that there exists $\tau >0$ such that the amplitude satisfies : 
$$||\alpha_h(z)||_\infty \leq  e^{- \tau \frac{\im z}{h} }$$
In fact, as explained in \cite{Nonnen11} (Theorem 6), microlocally near the trapped set, it is possible to write $$M^+(z) = M^+(0) \op(a_{h,z}) + O(h^{1-\varepsilon}) \; ; \; a_{h,z}(\rho) = \exp\left(\frac{ iz t(\rho)}{h} \right)$$ where $t(\rho)$ is the time needed for a ray emanating from $\rho$ to hit another obstacle. This fact is a consequence of the microlocal analysis performed in \cite{Ge88} (see Appendix II for the construction of a parametrix and III.2 for precise computations near the unique trapped ray for two obstacles, see also \cite{StefanovVodev95}). In particular, $\tau$ in the estimate above is a maximal return time for the billiard flow, in a small neighborhood of $\mathcal{T}$. 

Now, let $\gamma, \delta$ be the constants given by Proposition \ref{Prop_Crucial_Estimate}, depending, in this context, on the dynamics of the billiard map. 
Let us introduce the following threshold $\gamma_{\lim} = \frac{1}{2\tau} \frac{\gamma}{\delta} $ so that 
$$ z \in \Omega(\gamma_{\lim}, K,h) \implies ||\alpha_h(z)||_\infty \leq e^{\gamma/2\delta} < e^{ \gamma/\delta} $$
Proposition \ref{Prop_Estimate_M} now gives for $z \in \Omega(\gamma_{\lim}, K,h)$, 
$$\left| \left| \left( \Id - \mathcal{M}_g^+(z) \right)^{-1} \right| \right| \leq 2 \delta |\log h | h^{- \delta \log A }$$
where
$$A  \coloneqq \max(1, \tau \gamma_{\lim}) $$ 
Indeed, $A \geq 1$ and it allows to have $||\alpha_h(z)||_\infty \leq A$ for $z \in \Omega(\gamma_{\lim}, K,h)$. 
Going back to $\mathcal{M}^+(z)$, we get that 
$$\left| \left| \left( \Id - \mathcal{M}^+(z) \right)^{-1} \right| \right| \leq C |\log h | h^{- \delta \log A - 2 C_0T  }$$
where the extra $h^{-2C_0T}$ comes from the norm of $e^{\pm \op(g)}$. 
We conclude with Lemma \ref{Lemma_bound_mathcalM} and the formula (\ref{Key_formula_from_M_to_R}), using the estimates of Lemma \ref{Lemma_gamma_j} and \ref{Lemma_H_j}. This gives for $h$ small enough and $z \in \Omega(\gamma_{\lim}, K,h)$, 
\begin{equation}\label{last_equation}
||R_\theta(z)||_{L^2 \to L^2} \leq C| \log h| h^{-1 - 2 C_0 T - \delta \log A}
\end{equation}

\subsection{Proof in the case of scattering by a potential. }
The treatment of scattering by a potential is different and relies on a reduction to Poincaré sections of the Hamiltonian flow, under the assumption that the trapped set is totally disconnected. 

\subsubsection{Assumptions.}\label{Section_potential_scattering} We refer the reader to \cite{NSZ11} (Section 2.1) for more general assumptions. Here, we simply consider a smooth compactly supported potential $V \in \cinfc(\R^2)$ and work with the semiclassical differential operator $-h^2 \Delta + V$. We fix an energy $E_0>0$ and consider 
$$P_h = -h^2 \Delta + V - E_0$$
We note $p(x,\xi) = \xi^2 + V - E_0$ and we assume that $0$ is not a critical energy of $p$, that is
$$ dp \neq 0 \text{ on } p^{-1}(0)$$ 
Let's note $H_p$ the Hamiltonian vector field associated with $p$ and $\Phi_t = \exp(tH_p)$ the corresponding Hamiltonian flow. The trapped set at energy $0$ is the set 
$$K_0 = \{ (x,\xi) \in p^{-1}(0),  \exists R >0, \forall t \in \R , \Phi_t(x,\xi) \in B(0,R) \} $$
It is a compact subset of $p^{-1}(0)$. Here are the two crucial assumptions : 
\begin{enumerate}[label= (\roman*)]
\item $\Phi_t$ is hyperbolic on $K_0$ ; 
\item $K_0$ is topologically one dimensional. 
\end{enumerate}

\subsubsection{The reduction of \cite{NSZ11}}
We recall the main ingredients of the reduction to open quantum maps performed in \cite{NSZ11}. The aim of the following lines is to explain their crucial Theorem 5.  

Let us note $$\mathcal{R}(\eta, M_0,h) = \{ z \in \C, |\re z| \leq \eta, |\im z| \leq M_0 h \}$$ 
Here, $M_0$ is fixed (but large). As in the case of obstacle scattering, we fix once and for all the cut-off function $\chi \in \cinfc(\R^2)$ (with $\chi= 1$ in a neighborhood of $\supp(V)$) and we consider a complex scaled version of $P(h)$, $P_\theta(h)$ whose eigenvalues coincide with the resonances in $\mathcal{R}(\eta, M_0,h)$ and such that $\chi (P_\theta-z)^{-1}\chi = \chi (P-z)^{-1}  \chi$ for $z \in \mathcal{R}(\delta,M_0,h)$. 
Note that the parameter $\theta$ chosen in \cite{NSZ11} depends on $h$. 

Here are now the crucial ingredients of the reduction. 
\begin{itemize}
\item \textbf{Poincaré sections}. There exist finitely many smooth contractible hypersurfaces $\Sigma_i \subset p^{-1}(0)$, $i=1, \dots, J$ with smooth boundary and such that 
\begin{center}
$$\partial \Sigma_i \cap K_0 = \emptyset \; ; \; \Sigma_i \cap \Sigma_k = \emptyset , k \neq i $$
$H_p$ is transversal to  $\Sigma_i $ uniformly up to the boundary
\end{center}
Moreover, for $1 \leq i \leq J$ and for every $\rho \in K_0$, there exists $t_-(\rho) <0$ and $i_-(\rho)$ (resp. $t_+(\rho)>0$ and $i_+(\rho)$) such that 
\begin{align*}
\Phi_{t_\pm(\rho) } (\rho) \in K_0 \cap \Sigma_{i_\pm(\rho)} \\
\Sigma \cap \{ \Phi_t(\rho) , t_-(\rho) < t < t_+(\rho) , t \in \R^* \} = \emptyset
\end{align*}
where we note $$\Sigma = \bigsqcup_{i=1}^J \Sigma_i$$
\begin{figure}
\includegraphics[width=10cm]{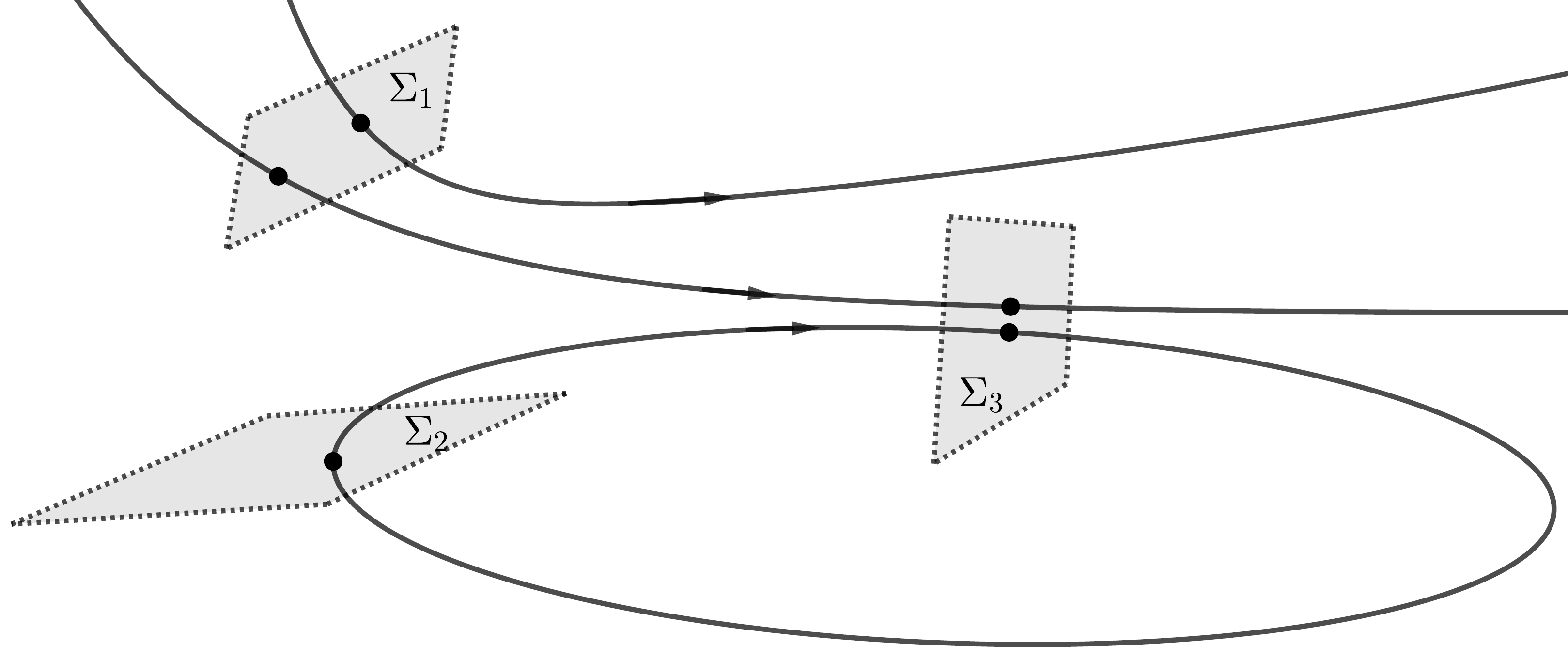}
\caption{Schematic representation of Poincaré sections for the flow $\Phi_t$ on an energy shell. The energy shell has dimension 3, so that the Poincaré section are 2-dimensional.  }
\label{figure_poincare_sections}
\end{figure}
The maps $t_\pm(\rho)$ are uniformly bounded on $K_0$ and can be smoothly extended in a neighborhood of $K_0$.  For convenience, it is also assumed that $i_+(\rho) \neq i$ for all $\rho \in K_0 \cap \Sigma_i$. This can be achieved by taking smaller and more Poincaré sections. Finally, there exist $\widetilde{\Sigma}_i \Subset \R^2$ and symplectic diffeomorphisms $$\kappa_i : \Sigma_i \to \widetilde{\Sigma}_i$$ smooth up to the boundary. 
\item \textbf{Poincaré return map}. For $1 \leq i,j \leq J$, the map $\rho \mapsto \Phi_{t_+(\rho)} (\rho)$, initially defined on $K_0$, extends smoothly to a symplectic diffeomorphism 
$$ F_{ij} : D_{ij} \to A_{ij}$$ by taking the intersection of the flow of a point $\rho \in D_{ij}$ with $\Sigma_i$
where $D_{ij}$ (resp. $A_{ij}$) is a neighborhood of $$ \{ \rho \in \mathcal{T} \cap \Sigma_j, i_+(\rho)=i\}  \text{ (resp. } \{ \rho \in \mathcal{T} \cap \Sigma_i, i_-(\rho)=j\} )$$ in $\Sigma_j$ (resp. $\Sigma_i$).  
The map $F_{ij}$ is called the Poincaré return map. By writing it in the charts $\kappa_i$ and $\kappa_j$, we can consider the following map between open sets of $T^*\R$ 
$$ \widetilde{F}_{ij} = \kappa_i \circ F_{ij} \circ \kappa_{j}^{-1} : \widetilde{D}_{ij} \subset \widetilde{\Sigma}_j \to  \widetilde{A}_{ij}\subset \widetilde{\Sigma}_i $$
Using the continuity of the flow, the same objects can be defined on energy shells $p^{-1}(z)$ for $z \in [-\delta, \delta]$ with $\delta$ small enough and we will note these objects $\Sigma_i^{(z)}$, $F_{ij}^{(z)}$, etc. In fact, it is possible to use the same open sets $\widetilde{\Sigma}_i$ and define, 
$$ \widetilde{F}_{ij}^{(z)} = \kappa_i \circ F_{ij}^{(z)} \circ \kappa_{j}^{-1} : \widetilde{D}_{ij}^{(z)} \subset \widetilde{\Sigma}_j \to  \widetilde{A}_{ij}^{(z)} \subset \widetilde{\Sigma}_i $$
The hyperbolicity of the flow implies the hyperbolicity of these open maps.

\item \textbf{Open quantum maps. } The notion of open quantum hyperbolic map associated with $\tilde{F}$ has been given in Definition \ref{def_FIO}. Since $\widetilde{\Sigma}_j \subset T^* \R$, we will simply say that it is an operator-valued matrix $T= (T_{ij})_{1 \leq i,j \leq J} : L^2(\R)^J \to L^2(\R)^J$ with $T_{ij} \in I_\delta(\R \times \R, \Gr\left( \widetilde{F}_{ij}\right)^\prime)$. In \cite{NSZ11}, the authors construct a particular family of open quantum hyperbolic maps, called $\mathcal{M}(z)$, where $\mathcal{M}(z)$ is associated with $\widetilde{F}^{(\re z)}$. This family is first microlocally defined near the trapped set and satisfies uniformly in $\mathcal{R}(\eta,M_0,h)$ and microlocally in a fixed neighborhood of the trapped set : 
$$\mathcal{M}_{ij}(z) = \mathcal{M}_{ij}(\re z) \op (a_z)+ O(h \log h) \; ; \; a_z = \exp\left( i \frac{ (\im z) t_+^{(z)}}{h} \right) $$
This particular family is built to solve a microlocal Grushin problem (see section 4 in \cite{NSZ11}). 

\item \textbf{Escape functions.} To perform a global study (i.e. no more microlocal) and to make the amplitude of $\mathcal{M}(z)$ very small outside a fixed neighborhood of the trapped set, the authors introduce an escape function for the flow $\Phi_t$, denoted $G_0 \in \cinfc(T^*\R^2)$ (see Lemma 5.3 in \cite{NSZ11}, where it is chosen independent of the energy variable near $\Sigma$). Let us note $G = Mh\log(1/h) G_0$, $g_j = G|_{\Sigma_j} \circ \kappa_j^{-1}$ and $$g : \bigsqcup_{j=1}^J \widetilde{\Sigma}_j \to \R, \rho \in \widetilde{\Sigma}_j \mapsto g_j(\rho) $$
$g$ is an escape function for the map $\widetilde{F}$. Each $g_j$ can be extended as an element of $\cinfc(\R^2)$ (see equation 5.2 and below in \cite{NSZ11}). 
\item \textbf{Conjugated operators.}
As in the case of obstacle scattering, we can consider the operators $$e^{\pm \op(G)} : L^2(\R^2) \to L^2(\R^2) \; ; \; e^{\pm \op(g)} : L^2(\R)^J \to L^2(\R)^J $$ 
Again, their norm is bounded by $O(h^{-K_G})$ for some $K_G>0$ depending on $G_0$ and $M$.  
We now introduce the following conjugated operators : 
$$ P_{\theta,G} = e^{-\op(G)}P_\theta e^{\op(G)} \; ; \, \mathcal{M}_{ij,g}(z) = e^{-\op(g_i)} \mathcal{M}_{ij}(z)e^{\op(g_j)} $$
$$\mathcal{M}_g(z) = \left( \mathcal{M}_{ij,g}(z) \right)_{1\leq i,j \leq J} $$ 
The escape function $G_0$ is built so that $\mathcal{M}_g(z) :L^2(\R)^J \to L^2(\R)^J$ is $O(h^{K_0})$ for some (large) $K_0$ microlocally outside a small neighborhood of the trapped set $\mathcal{T}$. In particular, $\mathcal{M}_g(z)$ satisfies the assumptions of the propositions \ref{Prop_Crucial_Estimate} and \ref{Prop_Estimate_M}. 

\item \textbf{A finite dimensional space.} For practical and technical reasons,\footnote{Mainly, to ensure the existence of determinants without discussion}  the authors choose to work with a finite dimensional version of the open quantum map $\mathcal{M}_g(z)$. To do so, they introduce finite rank projections $\Pi_j  : L^2(\R) \to L^2(\R)$ and the finite dimensional subspace of $L^2(\R)^J$
$$\mathcal{H} = \text{Ran } \Pi_1 \times  \dots \times \text{Ran } \Pi_J $$ 
The $\Pi_j$'s are built so that the projector $\Pi = \text{Diag}(\Pi_1, \dots, \Pi_J)$ satisfies the very important relation 
\begin{equation}\label{equivalent_finite_rank}
\Pi \mathcal{M}(z) \Pi = \mathcal{M}(z) \Pi + O(h^{K_1})
\end{equation}
for some large $K_1$ (in particular $K_1 \gg K_G$ so that the same relation holds after conjugation by $e^{\op g}$ with $K_1$ replaced by $K_1-2K_G$). We will note $\Pi_g = e^{-\op(g)} \Pi e^{\op(g)}$. 

\item \textbf{The Grushin problem.} 
To obtain a global Grushin problem (see section 5 in \cite{NSZ11}) the authors construct global operators $R_+(z) : H^2_h(\R^2) \to \mathcal{H}, R_-(z) : \mathcal{H} \to H^2_h(\R^2)$ which depend holomorphically on $z \in \mathcal{R}(\eta,M_0,h)$. The Grushin problem concerns \begin{equation} \label{grushin_problem}
\mathcal{P}_g(z) = \left( \begin{matrix}
P_{\theta,G} -z  & R_- (z) \\
R_+(z)& 0
\end{matrix} \right)  : H^2_h(\R^2) \times \mathcal{H} \to L^2(\R^2) \times \mathcal{H} 
\end{equation}
The goal of such a Grushin problem is to transform the eigenvalue equation $P_{\theta,G} u =zu$ into an equation on a simpler operator $E_\pm(z)$. This transformation is possible when $\mathcal{P}_g(z)$ is invertible. Indeed, in virtue of the so-called Schur complement formula, if $\mathcal{P}_g(z)$ is invertible with inverse $$\mathcal{E}(z) = \left( \begin{matrix}
E(z) & E_+(z)\\
E_-(z) &E_{\pm}(z) 
\end{matrix} \right)  : L^2(\R^2) \times \mathcal{H} \to H^2_h(\R^2) \times \mathcal{H}, $$
then $P_{\theta,G} -z $ is invertible if and only if $E_\pm(z)$ is and we have 
$$ (P_{\theta,G} - z)^{-1} = E(z) - E_+(z) E_\pm(z)^{-1}  E_-(z)$$
\end{itemize}

The authors prove the following result : 
\begin{thm}[\cite{NSZ11}, Theorem 5]
The Grushin problem (\ref{grushin_problem}) is invertible for all $z \in \mathcal{R}(\eta, M_0,h)$. If we note $$\mathcal{E}(z) = \left( \begin{matrix}
E(z) & E_+(z)\\
E_-(z) &E_{\pm}(z) 
\end{matrix} \right)  : L^2(\R^2) \times \mathcal{H} \to H^2_h(\R^2) \times \mathcal{H} $$ 
the inverse of $\mathcal{P}_g(z)$, then 
\begin{itemize}
\item\footnote{the norms are associated with the spaces mentionned above. For instance, $||E_-||_{\mathcal{H} \to H^2_h(\R^2)}$.}  $||E||, ||E_+||, ||E_-|| , ||E_\pm|| =O(h^{-1})$ uniformly in $\mathcal{R}(\eta, M_0,h)$.  
\item The operator $E_\pm(z)$ takes the form, for some $L_2 >0$
$$E_\pm(z) = \Id - M_g(z,h) + O(h^{L_2}) \; ; \; M_g(z,h) \coloneqq \Pi_g \mathcal{M}_g (z)\Pi_g$$
\end{itemize}
\end{thm}
\begin{rem}
As explained after Theorem 2 in \cite{NSZ11}, $L_2 = c^\prime M$ for some $c^\prime$, where $M$ is the one in the definition of the escape function $G$. In particular, $M$ can be chosen arbitrarily large, independently of $c^\prime$, so that $L_2$ can be made as large as necessary. 
\end{rem}

\subsubsection{End of proof. }
To rigorously apply Proposition \ref{Prop_Crucial_Estimate} to $\mathcal{M}_g(z)$, we fix $\eta_0 \in[-\eta, \eta]$ for $\eta$ small enough and consider $z \in D(\eta_0, Kh)$ for some fixed $0>K < M_0$. For such $z$, $\mathcal{M}(z)$ is an open quantum map associated with the Poincaré return map between the Poincaré sections $\Sigma(\eta_0) = \bigsqcup_{1 \leq j \leq J} \Sigma_j(\eta_0)$ inside the energy shell $p^{-1}(\eta_0)$. \\
Since $\mathcal{M}_g(z)$ satisfies the assumption of Proposition \ref{Prop_Crucial_Estimate} for $z \in D(\eta_0,Kh) \cap \mathcal{R}(\eta,M_0,h)$, it also satisfies its conclusion : 
\begin{equation}\label{Conclusion}
\exists h_0>0, \delta, \gamma >0 \; ; \; \forall 0 < h \leq h_0,  ||\mathcal{M}(z)_g^{N(h)}||_{L^2 \to L^2} \leq h^\gamma ||\alpha_h(z)||_\infty^{N(h)}
\end{equation} 
with $N(h) \sim \delta | \log h|$. \emph{A priori}, $h_0$, $\delta$, $\gamma$ depend on $\eta_0$. Nevertheless, as explained after Proposition \ref{Prop_Crucial_Estimate}, $\delta$ and $\gamma$ depend only on the properties of the Poincaré return map $F^{(\eta_0)}$ and $h_0$ depends on semi-norms of $\alpha_h(z) = e^{i\frac{z}{h} t_+^{(\eta_0)}}$. As explained in Section 4.1.1 in \cite{NSZ11}, this dynamics depends continuously on $\eta_0$ in a neighborhood of $0$ : that is, the departure sets, the arrival sets, the Poincaré maps and the return time function depend continuously on $\eta_0$. As a consequence, we can find $\eta$ and constants $\delta$, $\gamma$, $h_0$ such that 
(\ref{Conclusion}) holds for $z \in \Omega(\eta, \gamma,h) \coloneqq \{ |\re z | \leq \eta, -\gamma h \leq \im z\leq 0 \}$.  From this, we see that for $0 < h \leq h_0$ and for $z \in \Omega(\eta, \gamma,h)$,
 $$ ||\mathcal{M}(z)_g^{N(h)}||_{L^2 \to L^2} \leq h^\gamma e^{-N(h)\tau  \frac{\im z }{h} } \; ; \; \tau = \sup_{ |\eta_0| \leq \eta} ||t_\infty^{(\eta_0)}||_\infty$$
From (\ref{equivalent_finite_rank}), we see that for $N(h) \sim \delta |\log h|$, $$M_g(z)^{N(h)} = \Pi_g \mathcal{M}_g^{N(h)} \Pi_g + O( (\log h) h^{K_1})$$ so that we deduce that $M_g(z)$ satisfies also the conclusion of Proposition \ref{Prop_Crucial_Estimate} and hence, of Proposition \ref{Prop_Estimate_M}.  As a direct consequence, we obtain that for $h$ small enough, $E_\pm(z)$ is invertible for all $z \in \Omega(\eta, \gamma,h) $ and it satisfies for some $\beta >0$ : 

$$ || E_\pm (z)^{-1}||_{\mathcal{H} \to \mathcal{H}} \leq h^{-\beta}$$

We now conclude the proof as in the case of obstacle scattering, essentially replacing the formula (\ref{Key_formula_from_M_to_R}) by the standard Schur complement formula for the Grushin problem above : $E_\pm(z)$ is invertible if and only if $P_\theta - z$ is and 
$$ (P_{\theta,G} - z)^{-1} = E(z) - E_+(z) E_\pm(z)^{-1}  E_-(z)$$
Then, for $h$ small enough and for $z \in \Omega(\eta, \gamma, h)$, 
$$||(P_{\theta,G} - z)^{-1}||_{L^2 \to H_h^2} = O(h^{-1}) + O(h^{-\beta -2}) = O(h^{-\beta-2})$$
which gives 
$$||(P_{\theta} - z)^{-1}||_{L^2 \to H_h^2} = O(h^{-\beta-2 - 2K_G})$$
where $c_0$ depends on $G_0$.

\section{Application to the local energy decay for the wave equation}\label{Section_decay_energy}
We present an application of the resolvent estimate  obtained in the case of obstacle scattering to the decay of the local energy for the wave equation outside the obstacles. In this note, we follow the main arguments of \cite{Burq99} to prove Theorem \ref{Thm3}. 

\subsection{Resolvent estimates }

Let us rewrite the resolvent estimate of Theorem \ref{Thm2} in term of $\lambda$ : there exists $\gamma >0$, $\lambda_0 >0$ and $\beta >0$ such that for any $\chi \in \cinfc(\R^2)$ equal to one in a neighborhood of $\overline{\mathcal{O}}$, there exists $C_\chi >0$ such that for all $ \lambda \in \C$, 
\begin{equation}
 |\lambda| \geq \lambda_0, \im \lambda \geq - \gamma \implies ||\chi R(\lambda) \chi ||_{L^2 \to L^2} \leq C_\chi |\lambda|^\beta
\end{equation}
Recalling that for $f \in L^2_{comp}$, with $g= R(\lambda)f$ it holds that $\chi g \in H_0^1 (\Omega)$ and $g$ satisfies $-\Delta g + \lambda^2 g = f$, it is not hard to see that the above estimate implies that 
\begin{equation}
 ||\chi R(\lambda) \chi ||_{L^2 \to H^1_0} \leq C^\prime_\chi |\lambda|^{\beta+1}
\end{equation}
for $|\lambda| \geq \lambda_0$ and $\im \lambda \geq - \gamma$ (see for instance the proof of Proposition 2.5 in \cite{BuGeTz}). 

This gives resolvent estimates for large $\lambda$. We will also need to control the resolvent for small $\lambda$, in angular neighborhoods of the logarithmic singularity at $0$. For this purpose, we state a consequence of a result proved in \cite{Burq99} (Appendix B.2) : 

\begin{lem}
For $\varepsilon >0$, let $\mathcal{S}_\varepsilon = \{ \lambda \in \C^*, |\lambda| \leq \varepsilon, \arg \lambda \in [-\pi/4, 5\pi/4] \}$. \\
There exists $\varepsilon_0>0$ such that there is no resonance in $\mathcal{S}_{\varepsilon_0}$ and for any $\chi \in \cinfc(\R^2)$ equal to one in a neighborhood of $\overline{\mathcal{O}}$, there exists $C_\chi>0$ such that for all $\lambda \in \mathcal{S}_{\varepsilon_0}$, 
\begin{equation}
||\chi R(\lambda) \chi ||_{L^2 \to H^1_0} \leq C_\chi 
\end{equation}
\end{lem}

Finally, we also mention the following result, proved in \cite{Burq99} (Appendix B.1), which will be used below : 

\begin{lem}
There are no real resonances (that is with $\arg(z)=0$ or $\pi$). 
\end{lem}

\subsection{The wave equation generator}

Let $H$ be the Hilbert space $\mathcal{H}(\Omega) \oplus L^2(\Omega)$, where $\mathcal{H}$ is the completion of $\cinfc(\Omega)$ with respect to the norm $||f||_{\mathcal{H}} =||\nabla f ||_{L^2(\Omega)}$\footnote{This choice of Hilbert space makes the wave propagator unitary on $H$, since the energy of a solution of the wave equation is its norm in $H$;  see \cite{TaylorII}, Chapter 9, Section 4} and let $A$ be the operator 
$$ A = \left( \begin{matrix}
0 & \Id \\
\Delta & 0 
\end{matrix} \right)$$
with domain $D(A) =( \mathcal{H} \cap H^2(\Omega)) \oplus H_0^1(\Omega)$. $A$ is maximal dissipative, so that Hille-Yosida theory allows to define the propagator $e^{tA}$ and for $(u_0,u_1) \in H$, the first component $u(t)$ of $t \mapsto e^{tA} (u_0,u_1)$ is the unique solution of the following Cauchy problem 
$$ \left\{ \begin{array}{l}
\partial_t^2 u - \Delta u = 0 \\
u|_{t=0} = u_0 \\
\partial_t u|_{t=0} = u_1 
\end{array} \right. $$
Note also that since $A$ is maximal dissipative, for $\xi$ with $\re(\xi) >0$, $A- \xi$ is invertible and 
\begin{equation}\label{justifies_convergence_integral}
 ||(A- \xi)^{-1}||_{H \to H} \leq |\re \xi|^{-1}
\end{equation}
The global energy of the solution is defined as 
$$ E(t) = \frac{1}{2}||(u(t),\partial_t u(t) ) ||_H^2 =  \frac{1}{2} \int_{\Omega} |\nabla u(t)|^2 + |\partial_t u(t)|^2 $$ 
It is conserved. If $K \Subset \R^2$, we also define the local energy in $K$ as 
$$ E_K(t) =  \frac{1}{2} \int_{K \cap \Omega} |\nabla u(t)|^2 + |\partial_t u(t)|^2$$
Note that, by Poincaré inequality, if $B \subset \overline{\Omega}$ is bounded and if $\chi \in \cinfc(\R^2)$ is equal to one in a neighborhood of $\overline{\mathcal{O}}$ and is supported in $B$, then for $f \in \mathcal{H}(\Omega)$, 
$$ ||\chi f||_{\mathcal{H}} \sim ||\chi f ||_{H^1(B)} \sim ||\chi f||_{L^2(B)} + || \nabla (\chi f)||_{L^2(B)}$$
If $\chi \in \cinfc(\R^2)$ is equal to one in a neighborhood of $\overline{\mathcal{O}}$, by abuse we note $\chi$ the bounded operator of $(u,v) \in H \mapsto (\chi u, \chi v) \in H$. 

A short computation shows that for $\lambda \in \C$, $(u_0, u_1) \in D(A)$ and $(v_0,v_1) \in H$, 

$$ (A + i \lambda) \left( \begin{matrix} u_0 \\u_1  \end{matrix} \right) =\left( \begin{matrix} v_0 \\v_1  \end{matrix} \right)  \iff \left\{ \begin{array}{l}
(- \Delta - \lambda^2) u_0 = i \lambda v_0 -v_1 \\
u_1 = v_0 - i \lambda u_0 
\end{array} \right.$$
This relation and the remark above for bounded sets $B$ show that for any $\chi \in \cinfc(\R^2)$, the cut-off resolvent $\chi (A+ i \lambda)^{-1} \chi$, well defined for $\im \lambda >0$ extends to the logarithmic cover $\Lambda$ of $\C$ and we have for $\lambda \in \Lambda$ 

\begin{equation}
\chi (A+ i\lambda)^{-1} \chi = \left( \begin{matrix}
i \lambda \chi R(\lambda) \chi & - \chi R(\lambda) \chi \\
\chi^2 + \lambda^2 \chi R(\lambda) \chi & i \lambda \chi R(\lambda) \chi 
\end{matrix} \right) 
\end{equation}
We deduce that $\chi (A + i \lambda)^{-1} \chi $  has no real resonance and satisfies the following resolvent estimates, for some constant $C_\chi$, 
\begin{align}
\label{estimate_log_sing}\lambda \in \mathcal{S}_{\varepsilon_0}& \implies  \left| \left| \chi (A + i \lambda)^{-1} \chi \right| \right|_{H \to H} \leq C_{\chi} \\
\label{estimate_strips} |\lambda| \geq \lambda_0 ,  \im \lambda \geq - \gamma  &\implies  \left| \left| \chi (A + i \lambda)^{-1} \chi \right| \right|_{H \to H}  \leq C_{\chi}|\lambda|^{\beta+2} 
\end{align}

\subsection{Proof of the local energy decay}
Let us fix $R >0$ such that $\mathcal{O} \Subset B(0,R)$. We want to estimate the local energy in $B(0,R)$ for solutions with initial data supported in $B(0,R)$, and sufficiently regular, that is in $D(A^k)$ for a sufficiently large $k$. As we will see, the decay will hold for data in $D(A^k)$ with $k \geq \beta + 4$, where $\beta$ is the one appearing in (\ref{estimate_strips}). For this purpose, let us fix $\chi \in \cinfc(\R^2)$ such that $\chi = 1$ in $B(0,R)$. 

Let $U_0  \in D(A^k)$ with $\supp(U_0) \Subset B(0,R)$. We want to estimate the energy of $\chi e^{tA} U_0$, or equivalently, we want to control $\left| \left| \chi e^{tA}  U_0\right|\right|_H$. Let us write $U = (I-A)^{k}U_0 \in H$, so that $ ||U_0||_{D(A^k)}=||U||_H$. It is clear that we have $\supp (U) \subset B(0,R)$, so that $U = \chi U$. With this notation, we want to show that there exists $C_R >0$ such that for all $t \geq 1$, 
$$ I(t) \coloneqq \left| \left| \chi e^{tA} (I-A)^{-k} \chi U  \right| \right|_H \leq \frac{C}{t}||U||_H$$

The starting point of the proof is the following formula : 

\begin{lem}
Assume that $k \geq 2$. 
For $t \geq 0$ and for $U \in H$, we have 
\begin{equation}\label{Integral}
e^{tA} (\I-A)^{-k}U = \frac{-1}{2\pi} \int_{ \lambda \in \frac{i}{2} + \R} e^{-it \lambda} \frac{1}{(1+i \lambda)^k} (A + i \lambda)^{-1} U d \lambda
  \end{equation}
\end{lem}

\begin{proof}
First remark that the integral $I(t)$ in the right hand side is absolutely convergent in virtue of (\ref{justifies_convergence_integral}) and since $k \geq 2$. \\
Differentiating the right hand side with respect to $t$, we find that 
\begin{align*}
(\partial_t -A )I(t) &=  \frac{-1}{2\pi} \int_{ \lambda \in \frac{i}{2} + \R} e^{-it \lambda} \frac{1}{(1+i \lambda)^k} (-i \lambda -A)(A + i \lambda)^{-1} U d \lambda \\
&= \frac{1}{2\pi} \int_{ \lambda \in \frac{i}{2} + \R} e^{-it \lambda} \frac{1}{(1+i \lambda)^k} U d \lambda =0
\end{align*}
(To see that the last integral is equal to zero, one can for instance perform a contour deformation from $\im(\lambda) = 1/2$ to $\im (\lambda) = -\rho$ and let $\rho$ tend to $+ \infty$. )\\
Finally, we need to check that $I(0) = (\Id-A)^{-k} U$. We have $$I(0) =  \frac{-1}{2\pi} \int_{ \lambda \in \frac{i}{2} + \R}  \frac{1}{(1+i \lambda)^k} (A + i \lambda)^{-1} U d \lambda $$ 
We perform a contour deformation. Let $r >1$ and let $\Gamma_r$ be rectangle joining the points $i/2 + r , r(1+i) , r(i-1), i/2 - r$.  We also note $\gamma_r = \Gamma_r \setminus [-r + i/2, r + i/2]$. The function $g_k : z \mapsto -(1+iz)^{-k} (iz \Id + A)^{-1}U$ is meromorphic in $\im z >0$, with a unique pole at $z=i$. As a consequence, we find that 
\begin{align*}
\frac{1}{2i \pi} \int_{\Gamma_r} g_k(z) dz  = \text{Res}_{z=i} g_k = \frac{-1}{i^k (k-1)!} \partial_z^{k-1} ((iz\Id + A)^{-1})U |_{z=i} = i (\Id-A)^{-k} U 
\end{align*}
Hence, we have 
\begin{align*}
I(0) =\lim_{r \to + \infty} \frac{-i}{2i \pi } \int_{i/2 -r}^{i/2+r} g_k(\lambda) d\lambda = \lim_{r \to + \infty} -i\left( i(\Id - A)^{-k}U - \int_{\gamma_r} g_k(z) dz \right)=  (\Id - A)^{-k}U 
\end{align*} 
Indeed, it is not hard to see that the contribution on $\gamma_r$ tends to 0 as $r \to + \infty$. 
\end{proof}

\textit{End of proof of Theorem \ref{Thm3}}. The proof relies on a contour deformation below the real axis in the integral of (\ref{Integral}) where the cut-off $\chi$ is inserted, but we need to get around the logaritmic singularity at 0 : it is possible due to (\ref{estimate_log_sing}). 

We fix $t>0$. In the estimates below, the constants denoted by $C$ (or $C_k)$ do not depend on $t$. 
We know that the map $ \lambda \mapsto \chi (A+i\lambda)^{-1} \chi U $ is meromorphic in $\C \setminus i\R^-$, with no poles in $\{ \im \lambda >0 \} \cup \mathcal{S}_{\varepsilon_0} \cup \{ \im \lambda \geq - \gamma, |\lambda| \geq \lambda_0 \}$. By taking $\varepsilon_0$ smaller if necessary, we may assume that $2^{-1/2} \varepsilon_0  \leq \gamma$. Let $K$ be the union of the rectangles $K_+$ and $K_-$ where 
$$ K_{\pm} = \{ \lambda \in \C, \re \lambda \in [\pm 2^{-1/2} \varepsilon_0, \pm \lambda_0] , \im \lambda \in [-2^{-1/2} \varepsilon_0, 0] \}$$
Since there is only a finite number of resonance in $K$ and since there are no resonances on $K \cap \{ \im \lambda = 0 \}$, we can find $\delta >0$ such that there is no resonance in $K \cap \{ \im \lambda \geq - \delta\}$ and since this region is compact, we can find $C$ such that for $\lambda \in K \cap \{ \im \lambda \geq - \delta\}$, 
$$ || \chi (A+ i\lambda)^{-1} \chi||_{ H \to H } \leq C $$
Let's note $z^+$ (resp. $z^-$) the unique point in $\{ |z| = \varepsilon_0 \} \cap \{ \im z= - \delta \} \cap \{\pm \re z >0 \}$. Fix $r \gg 1$ and $0 < \varepsilon < \varepsilon_0$ and let's note $z_\varepsilon^\pm$ the point of the segment $[0, z^\pm]$ with norm $\varepsilon$ and let's introduce the following paths, oriented from the left point to the right point : 
\begin{align*}
&\gamma_r^+ = [z^+, r - i \delta] \; ; \; \gamma_r^-=[-r- i \delta, z^-] \\
& \gamma_\varepsilon^+ = [z^+_\varepsilon, z^+] \; ; \;  \gamma_\varepsilon^+ = [z^-, z_\varepsilon^-] \\
&l_r^+ =[r-i\delta, r +i/2] \; ; \; l_r^- = [-r-i\delta, -r+i/2] \\
&l_r = [-r+i/2, r+i/2]
\end{align*}
and $ \mathcal{C}_\varepsilon $ be the arc of the circle $\{|z| = \varepsilon\} $ from $z^-_\varepsilon$ to $z^+_\varepsilon$. (See Figure \ref{figure_contour}). 

\begin{figure}
\usetikzlibrary{arrows}
\usetikzlibrary{decorations.markings}
\tikzstyle directed=[postaction={decorate,decoration={markings,
    mark=at position .65 with {\arrow{latex}}}}]
\begin{tikzpicture}[line cap=round,line join=round,>=triangle 45,x=1cm,y=1cm]
\begin{axis}[
x=1cm,y=1cm,
axis lines=middle,
xtick=\empty, 
ytick=\empty,
xmin=-4,
xmax=4,
ymin=-1,
ymax=2,
yticklabels={}, xticklabels={}]
\clip(-4,-1.5) rectangle (4,2.5);
\draw [line width=0.5pt,directed] (0.5,-0.5)-- (3.5,-0.5);
\draw [line width=0.5pt,directed] (-3.5,1.5)-- (3.5,1.5);
\draw [line width=0.5pt,directed] (3.5,-0.5)-- (3.5,1.5);
\draw [line width=0.5pt,directed] (-3.5,1.5)-- (-3.5,-0.5);
\draw [line width=0.5pt,directed] (-3.5,-0.5)-- (-0.5,-0.5) ;
\draw [line width=0.5pt]  plot[domain=-0.7853981633974483:3.9269908169872414,variable=\t]({0.2*cos(\t r)},{0.2*sin(\t r)});
\draw [line width=0.5pt,directed] (-0.5,-0.5)-- (-0.14142135623730953,-0.14142135623730953);
\draw [line width=0.5pt,directed] (0.1414213562373095,-0.1414213562373095)-- (0.5,-0.5);
\begin{scriptsize}
\draw [fill=black] (0.5,-0.5) circle (1pt);
\draw [fill=black] (3.5,-0.5) circle (1pt);
\draw[color=black] (2,-0.5) node[below] {$\gamma_r^+$};
\draw [fill=black] (-3.5,1.5) circle (1pt);
\draw [fill=black] (3.5,1.5) circle (1pt);
\draw[color=black] (-0.6,1.5) node[above] {$l_r$};
\draw[color=black] (3.5,0.7) node[left] {$l_r^+$};
\draw [fill=black] (-3.5,-0.5) circle (1pt);
\draw[color=black] (-3.5, 0.6) node[right] {$l_r^-$};
\draw [fill=black] (-0.5,-0.5) circle (1pt);
\draw[color=black] (-2,-0.5) node[below] {$\gamma_r^-$};
\draw [fill=blue] (-0.14142135623730953,-0.14142135623730953) circle (2pt);
\draw [fill=red] (0.1414213562373095,-0.1414213562373095) circle (2pt);
\draw[color=black] (0,0.2) node[above right] {$\mathcal{C}_\varepsilon$};
\draw[color=black] (-0.2,-0.2) node[left] {$\gamma_\varepsilon^-$};
\draw[color=black] (0.2,-0.2) node[right] {$\gamma_\varepsilon^+$};
\end{scriptsize}
\end{axis}
\end{tikzpicture}
\caption{The contour used to deform the integral. $z^-$ (resp. $z^+$) is the blue (resp. red) point on the figure.}
\label{figure_contour}
\end{figure}

With $f_k(z) = -\frac{1}{2\pi} e^{-itz} (1+iz)^{-k} \chi (A+iz)^{-1} \chi U$, 
we have $\chi e^{tA} (\Id-A)^{-k} \chi U = \lim_{r \to + \infty} \int_{l_r} f_k(z) dz$ and since $f_k(z)$ is holomorphic in a neighborhood of the compact set surrounded by the above contours, we have 
$$ \int_{l_r} f_k(z) dz= \int_{l_r^+} f_k(z) dz+\int_{l_r^-} f_k(z) dz+\int_{\gamma_r^+} f_k(z) dz+\int_{\gamma_r^-} f_k(z) dz+\int_{\mathcal{C_\varepsilon}} f_k(z) dz+\int_{\gamma_\varepsilon^+} f_k(z) dz+\int_{\gamma_\varepsilon^-} f_k(z) dz$$
Note that $\mathcal{C}_\varepsilon \cup \gamma_\varepsilon^+ \cup \gamma_\varepsilon^- \subset \mathcal{S}_{\varepsilon_0}$. As a consequence, we have 
$$ \left| \left| \int_{\mathcal{C_\varepsilon}} f_k(z) dz\right|\right|_H \leq C_k  \varepsilon ||U||_H  \to_{ \varepsilon \to 0} 0$$
and, with $\theta = \arg z^+$, 
$$ \left| \left| \int_{\gamma_\varepsilon^+} f_k(z) \right| \right|_H \leq C_k  \int_{\varepsilon}^{\varepsilon_0} e^{t s \sin \theta} ||U||_H ds \leq C_k  \frac{||U||_H}{t |\sin \theta |} \left( e^{t \varepsilon \sin \theta} - e^{t \varepsilon_0 \sin \theta} \right) \leq \frac{C_k}{t} ||U||_H$$
The case of $\gamma_\varepsilon^-$ is treated similarly.\\
On $\gamma_r^\pm$, the following holds : 
$$ ||\chi (A+i\lambda)^{-1} \chi ||_{H \to H} \leq C|\lambda|^{\beta+2}$$ 
Indeed, this is true for $|\lambda| \geq \lambda_0$ and there is no resonance on $\gamma_r^\pm$. As a consequence, for $\lambda = -i \delta + \xi$, $|\xi| \geq \re (z_+)$, we have 
$$ ||\chi e^{-it \lambda} (1+i\lambda)^{-k} (A+ i \lambda)^{-1} \chi ||_{H \to H} \leq C e^{-t \delta} |\xi|^{\beta +2 - k}$$
 Hence, we assume here that $\boxed{k \geq \beta + 4}$ so that 
 $$\left| \left| \int_{\gamma_r^\pm} f_k(z)  dz \right| \right|_H \leq C \int_{\re(z^+)}^r e^{-t \delta}|\xi|^{\beta+2-k} ||U||_H d \xi \leq Ce^{-t \delta} \int_{\re(z^+)}^{+\infty} |\xi|^{-2} ||U||_H d\xi \leq Ce^{-t\delta} ||U||_H$$
Finally, we treat the vertical segments $l_r^\pm$. 
$$\left| \left| \int_{l_r^\pm} f_k(z)  dz \right| \right|_H \leq C \sup_{y \in [-\delta, 1/2]} ||f_k(\pm r + iy )||_H$$ 
For $y \in [-\delta,1/2]$, we have 
\begin{align*}
||f_k(\pm r + iy )||_H &\leq Ce^{ty} r^{-k} ||\chi(A+i (r + iy))^{-1} \chi ||_{H \to H} \times ||U||_H\\
& \leq C e^{t/2} r^{-k} ||\chi(A+i (r + iy))^{-1} \chi ||_{H \to H} \times ||U||_H
\end{align*}
Using (\ref{estimate_strips}), we find that for $ y \in [-\delta,1/2]$,  $||\chi(A+i (r + iy))^{-1} \chi ||_{H \to H} \leq Cr^{\beta+2}$. As a consequence, one finds that for $y \in [-\delta,1/2]$,  $$||f_k(\pm r + iy )||_H \leq Ce^{t/2}r^{\beta+2-k} ||U||_H \leq Ce^{t/2} r^{-2}  ||U||_H$$
As a consequence, 
$$ \left| \left| \int_{l_r} f_k(z)dz \right| \right|_H \leq \left( \frac{C_k}{t} + C e^{-t\delta} + C_k \varepsilon + Ce^{t/2} r^{-2} \right)||U||_H$$
By letting $ \varepsilon$ tending to 0 and $r$ to $+ \infty$, we conclude that 
$$ ||\chi e^{tA} (\Id-A)^{-k} \chi U||_H \leq \left( \frac{C_k}{t} + Ce^{-t \delta} \right)||U||_H \leq \frac{C_k}{t} ||U||_H $$
which gives the required result.

\appendix

\section{Tools of semiclassical analysis}\label{Appendix}
We review the most important notions of semiclassical analysis needed in this note. . 
\subsection{Pseudodifferential operators and Weyl quantization}
We recall some basic notions and properties of the Weyl quantization on $\R^n$. We refer the reader to \cite{ZW} for the proofs of the statements and further considerations on semiclassical analysis and quantizations. We start by defining classes of $h$-dependent symbols. 

\begin{defi}
Let $0 \leq \delta \leq \frac{1}{2}$. We say that an $h$-dependent family $a \coloneqq \left( a (\cdot ; h) \right)_{0 < h \leqslant 1}$ is in the class $S_\delta(T^*\R^n)$ (or simply $S_\delta$ if there is no ambiguity) if for every $\alpha \in \N^{2n}$, there exists $C_\alpha >0$ such that : 
$$\forall 0< h \leq 1, \sup_{(x,\xi) \in \R^n} | \partial^\alpha a (x,\xi ; h) | \leq C_\alpha h^{-\delta |\alpha|}$$
\end{defi}
In this paper, we will mostly be concerned with $\delta <1/2$. We will also use the notation $S_{0^+} = \bigcap_{\delta >0} S_\delta$. \\
We write $a =  O(h^N)_{S_\delta}$ to mean that for every $\alpha \in \N^{2n}$, there exists $C_{\alpha,N}$ such that 
$$\forall 0< h \leq 1, \sup_{(x,\xi) \in \R^n}   | \partial^\alpha a (x,\xi ; h) | \leq  C_{\alpha,N} h^{-\delta |\alpha|} h^N $$ 
 If $a=O(h^N)_{S_\delta}$ for all $N \in \N$, we'll write $a= O(h^\infty)_{S_\delta}$. 

For a given symbol $a \in S_\delta(T^*\R^n)$, we say that $a$ has a compact essential support if there exists a compact set $K$ such : 
$$ \forall \chi \in \cinfc(\R^n), \supp \chi \cap K = \emptyset \implies \chi a =\hinf_{\mathcal{S}(T^*\R^n) }$$
(here $\mathcal{S}$ stands for the Schwartz space). We note $ \text{ess} \supp a \subset K$ and say that $a$ belongs to the class $S_\delta^{comp}(T^*\R^n) $. The essential support of $a$ is then the intersection of all such compact $K$'s. In particular, the class $S_{\delta}^{comp}$ contains all the symbols in $S_\delta$ supported in a $h$-independent compact set and these symbols correspond, modulo $\hinf_{\mathcal{S}(T^*\R) }$, to all symbols of $S_\delta^{comp}$. For this reason, we will adopt the following notation : for an open set $\Omega \subset \R^n$,  $a \in S_\delta^{comp}(\Omega) \iff \text{ess} \supp a \Subset \Omega$.

For a symbol $a \in S_\delta(T^*\R^n)$, we will quantize it using Weyl's quantization procedure. It is written as : 
$$(\op(a)u)(x) = (a^W(x,hD_x) u)(x)= \frac{1}{(2\pi h)^n }\int_{\R^{2n}} a\left( \frac{x+y}{2}, \xi \right)u(y) e^{i \frac{(x-y)\cdot \xi}{h}}dy d \xi$$

We will note $\Psi_\delta(\R^n)$ the corresponding classes of pseudodifferential operators. By definition, the wavefront set of $A = \op(a)$ is $\WF(A) = \text{ess} \supp a $. 
\vspace*{0.5cm}

We say that a family $u=u(h) \in \mathcal{D}^\prime(\R^n)$ is $h$-tempered if for every $\chi \in \cinfc(\R^n)$, there exist $C >0$ and $N \in \N$ such that $|| \chi u ||_{H_h^{-N}} \leq Ch^{-N}$. For a $h$-tempered family $u$, we say that a point $\rho \in T^*\R^n$ does \emph{not} belong to the wavefront set of $u$ if there exists $a \in S^{comp}(T^*\R^n)$ such that $a(\rho) \neq 0$ and $\op(a)u = \hinf_{\mathcal{S}}$. We note $\WF(u)$ the wavefront set of $u$. 

We say that a family of operators $B=B(h) : \cinfc(\R^{n_2}) \to \mathcal{D}^\prime(\R^{n_1})$ is $h$-tempered  if its Schwartz kernel $\mathcal{K}_B \in \mathcal{D}^\prime(\R^{n_1} \times \R^{n_2})$ is $h$-tempered. The wavefront set of $B$, denoted $\WF^\prime(B)$ is defined as
$$ \WF^\prime(B) = \{ ( x,\xi, y , -\eta) \in T^*\R^{n_1}  \times T^* \R^{n_2} ,  (x,\xi, y , \eta)  \in \WF(\mathcal{K}_B )\}$$

The Calderon-Vaillancourt Theorem asserts that pseudodifferential in $\Psi_\delta$ are bounded on $L^2$ and as a consequence of the sharp Gärding inequality (see \cite{ZW}, Theorem 4.32), we also have a precise estimate of $L^2$ norms of pseudodifferential operator, 

\begin{prop}\label{Garding}
Assume that $a \in S_\delta(\R^{2n})$. Then, there exists $C_a$ depending on a finite number of semi-norms of $a$ such that : 
$$ || \op(a) ||_{L^2 \to L^2} \leq ||a||_\infty + C_ah^{\frac{1}{2} - \delta}$$
\end{prop}

We recall that the Weyl quantizations of real symbols are self-adjoint in $L^2$. The composition of two pseudodifferential operators in $\Psi_\delta$ is still a pseudodifferential operator. More precisely (see \cite{ZW}, Theorem 4.11 and 4.18), if $a , b \in S_\delta$, $\op(a) \circ \op(b)$ is given by $\op(a \# b)$, where $a \# b$ is the Moyal product of $a$ and $b$. It is given by 
$$ a \# b (\rho) = e^{i h A(D) } (a \otimes b)|_{ \rho_1= \rho_2 = \rho} $$ 
where $a \otimes b (\rho_1, \rho_2) = a(\rho_1) b(\rho_2)$, $e^{i h A(D) }$ is a Fourier multiplier acting on functions on $\R^{4n}$ and, writing $\rho_i = (x_i, \xi_i)$,  $$A(D) = \frac{1}{2} \left( D_{\xi_1} \circ D_{x_2} - D_{x_1} \circ D_{\xi_2} \right) $$

\subsection{Fourier Integral Operators}
We review some aspects of the theory of Fourier integral operators. We follow \cite{ZW}, Chapter 11 and \cite{NSZ14}. We refer the reader to \cite{GuSt} for further details or to \cite{Al08}. Finally, we will give the precise definition needed to understand the definition \ref{def_FIO}. 

\subsubsection{Local symplectomorphisms and their quantization}
Let us note $\mathcal{K}$ the set of symplectomorphisms $\kappa : T^*\R^n \to T^* \R^n$ such that the following holds : there exist continuous and piecewise smooth families of smooth functions $(\kappa_t)_{t \in [0,1]} $, $(q_t)_{t \in [0,1]}$ such that : \begin{itemize}[nosep]
\item $\forall t \in [0,1]$, $\kappa_t : T^*\R^n \to T^*\R^n$ is a symplectomorphism ; 
\item $\kappa_0 = \Id_{T^* \R^n} , \kappa_1 = \kappa$ ; 
\item $\forall t \in [0,1], \kappa_t(0)=0 $ ; 
\item there exists $K \Subset T^*\R^n$ compact such that $\forall t \in [0,1], q_t : T^*\R^n \to \R$ and $ \supp q_t \subset K$ ; 
\item $\frac{d}{dt} \kappa_t = \left( \kappa_t \right)^\star H_{q_t}$
\end{itemize}
We recall  \cite{ZW}, Lemma 11.4, which asserts that local symplectomorphisms can be seen as elements of $\mathcal{K}$, as soon as we have some geometric freedom. 

\begin{lem}\label{lemma_local_symp}
Let $U_0, U_1$ be open and precompact subsets of $T^*\R^n$. Assume that $\kappa : U_0 \to U_1$ is a local symplectomorphism   that extends to $V_0 \Supset U_0$ an open star-shaped set. Then, there exists $\tilde{\kappa} \in \mathcal{K}$ such that $\tilde{\kappa}|_{U_0} = \kappa$. 
\end{lem}

 If $\kappa \in \mathcal{K}$ and  if $(q_t)$ denotes the family of smooth functions associated with $\kappa$ in its definition, we note $Q(t) = \op(q_t)$. It is a continuous and piecewise smooth family of operators. Then the Cauchy problem
 
 \begin{equation}\label{Cauchy_pb_Egorov}
\left\{ \begin{array}{c}
hD_t U(t) + U(t)Q(t) = 0 \\
U(0)=\Id
\end{array}
\right. 
\end{equation}
is globally well-posed.

From now on, we restrict to the case $n=1$. Following \cite{NSZ14}, Definition 3.9, we adopt the definition : 

\begin{defi}
Let $\kappa \in \mathcal{K}$ and let us note $C = Gr^\prime(\kappa) = \{ (x, \xi, y , - \eta), (x,\xi) = \kappa(y,\eta) \}$ the twisted graph of $\kappa$. \\
Fix $\delta \in [0,1/2)$. We say that $U \in I_\delta(\R \times \R; C )$ if there exists $a \in S_\delta(T^*\R)$ and a path $(\kappa_t)$ from $\Id$ to $\kappa$ satisfying the above assumptions such that $U = \op(a)U(1)$, where $t \mapsto U(t)$ is the solution of the Cauchy problem (\ref{Cauchy_pb_Egorov}). 

The class $I_{0^+}(\R \times \R, C)$ is by definition $\bigcap_{\delta >0} I_\delta(\R \times \R, C)$. 
\end{defi}

It is a standard result, known as Egorov's theorem (see \cite{ZW}, Theorem 11.1) that if $U(t)$ solves the Cauchy problem (\ref{Cauchy_pb_Egorov}) and if $a \in S_\delta$, then $U^{-1} \op(a) U$ is a pseudodifferential operator in $\Psi_\delta$ and if $b= a \circ \kappa$, then 
$U^{-1} \op(a) U - \op(b) \in h^{1- 2 \delta} \Psi_\delta$.

\begin{rem}\text{}
Applying Egorov's theorem and Beals's theorem, it is possible to show that if $(\kappa_t)$ is a closed path from $\Id$ to $\Id$, and $U(t)$ solves (\ref{Cauchy_pb_Egorov}), then $U(1) \in \Psi_0(\R^n)$. In other words, $I_\delta(\R \times \R , \Gr^\prime(\Id) ) \subset \Psi_\delta(\R^n)$. But the other inclusion is trivial. Hence, this in an equality : 
$$I_\delta(\R \times \R , \Gr^\prime(\Id) ) = \Psi_\delta(\R^n)$$
The notation $I(\R\times \R, C)$ comes from the fact that the Schwartz kernels of such operators are Lagrangian distributions associated to $C$, and in particular have wavefront sets included in $C$. As a consequence, if $T \in I_\delta(\R\times \R, C)$, $\WF^\prime(T) \subset \Gr(\kappa)$. 
\end{rem}

We also recall that the composition of two Fourier integral operators is still a Fourier integral operator : if $\kappa_1, \kappa_2 \in \mathcal{K}$ and $T_1 \in I_\delta(\R \times \R, \Gr^\prime(\kappa_1) ) , T_2 \in I_\delta(\R \times \R, \Gr^\prime(\kappa_1) )$, then, 
$$ T_1 \circ T_2 \in  I_\delta(\R \times \R, \Gr^\prime(\kappa_1 \circ \kappa_2) )$$ 

\subsubsection{Quantization of open symplectic maps}\label{appendix_symplectic_map_quantization}
As in Section \ref{Section_2}, we consider a symplectic map $F$ which is the union of local open symplectic $F_{ij} : \widetilde{D}_{ij} \subset U_j \to \widetilde{A}_{ij} \subset U_i$, where $U_i \subset T^*\R$ are open sets.  We keep the same notations. In particular, $\mathcal{T}$ is the trapped set and the full arrival (resp. departure) set is $\widetilde{A}$ (resp. $\widetilde{D}$). 
We fix a compact set $W \subset \tilde{A}$ containing some neighborhood of $\mathcal{T}$. Our definition will depend on $W$ and, is not, in some sense, canonical. 
Following \cite{NSZ14} (Section 3.4.2), we now focus on the definition of the elements of $I_\delta(Y \times Y; \Gr(F)^\prime)$. An element $T \in I_\delta(Y \times Y; \Gr(F)^\prime)$ is a matrix of operators  
$$T = (T_{ij})_{1 \leq i,j\leq J } : \bigoplus_{j=1}^J L^2(Y_j) \to \bigoplus_{i=1}^J L^2(Y_i)$$
Each $T_{ij}$ is an element of $I_\delta(Y_i \times Y_j , \Gr(F_{ij})^\prime )$. Let's now describe the recipe to construct elements of $I_\delta(Y_i \times Y_j, \Gr(F_{ij})^\prime )$. 

\begin{itemize}
\item Fix some small $\varepsilon >0$ and two open covers of $U_j$, $U_j\subset \bigcup_{l=1}^L \Omega_l$, $\Omega_l \Subset \widetilde{\Omega}_l$, with $\widetilde{\Omega}_l$ star-shaped and having diameter smaller than $\varepsilon$. We note $\mathcal{L}$ the sets of indices $l$ such that $\Omega_l \subset \widetilde{D}_{ij} \subset U_j$ and we require (this is possible if $\varepsilon$ is small enough)  
$$ F^{-1}(W) \cap U_j \subset \bigcup_{l \in \mathcal{L}} \Omega_l $$

\item Introduce a smooth partition of unity associated to the cover $(\Omega_l)$, $(\chi_l)_{1 \leq l \leq L} \in \cinfc(\Omega_l, [0,1])$, $\supp \chi_l \subset \Omega_l$, $\sum_l \chi_l = 1$ in a neighborhood of $\overline{U_j}$.
\item  For each $l \in \mathcal{L}$, we denote $F_l$ the restriction to $\widetilde{\Omega}_l$ of $F_{ij}$. By Lemma \ref{lemma_local_symp}, there exists $\kappa_l \in \mathcal{K}$ which coincides with $F_l$ on $\Omega_l$. 
\item  We consider $T_l=\op(\alpha_l) U_l(1)$ where $U_l(t)$ is the solution of the Cauchy problem (\ref{Cauchy_pb_Egorov}) associated to $\kappa_l$ and $\alpha_l \in S_\delta^{comp}(T^*\R)$. 
\item We set 

\begin{equation}\label{def_FIO_global}T^\R = \sum_{l \in \mathcal{L}} T_l \op(\chi_l) : L^2(\R) \to L^2(\R)
\end{equation} 
$T^\R$ is a globally defined Fourier integral operator. We will note $T^\R \in I_\delta(\R \times \R, \Gr(F_{ij})^\prime)$. Its wavefront set is included in $\widetilde{A}_{ij} \times \widetilde{D}_{ij}$. 
\item Finally, we fix cut-off functions $(\Psi_i, \Psi_j) \in \cinfc(Y_i, [0,1]) \times \cinfc(Y_j, [0,1])$ such that $\Psi_i \equiv 1$ on $\pi(U_i)$ and $\Psi_j \equiv 1$ on $\pi(U_j)$(here, $\pi : (x,\xi) \in T^*Y_{\cdot} \mapsto  x \in Y_{\cdot}$ is the natural projection) and we adopt the following definitions : 
\end{itemize}

\begin{defi}\label{Def_FIO_local}
We say that $T : \mathcal{D}^\prime(Y_j) \to \cinf(\overline{Y_i})$ is a Fourier integral operator in the class $I_\delta(Y_i \times Y_j, \Gr(F_{ij})^\prime)$  if there exists $T^\R \in I_\delta(\R \times \R, \Gr(F_{ij})^\prime)$ as constructed above such that
\begin{itemize}
\item $T - \Psi_i T \Psi_j = \hinf_{\mathcal{D}^\prime(Y_{j}) \to \cinf(\overline{Y}_i)}$; 
\item $\Psi_i T \Psi_j =  \Psi_i T^\R \Psi_j $
\end{itemize}
\end{defi}

For $U^\prime_j \subset U_j$ and $U_i^\prime = F(U_j^\prime) \subset U_i$, we say that $T$ (or $T^\R$)  is microlocally unitary in $U_i^\prime \times U_j^\prime$ if $TT^* = \Id$ microlocally in $U_i^\prime$ and $T^*T = \Id$ microlocally in $U_j^\prime$.

\bibliographystyle{alpha}
\bibliography{biblio_these}

\end{document}